\documentclass[twoside,a4paper,reqno,11pt]{amsart} 
\usepackage{amsfonts, amsbsy, amsmath, amssymb, latexsym,hyperref, mathtools, bold-extra, url}
\usepackage{mathrsfs,array}
\usepackage[top=24mm,right=28mm,bottom=24mm,left=28mm]{geometry}
\usepackage{stmaryrd}
\usepackage{bm}
\usepackage[pdftex]{color,graphicx}

\renewcommand{\a}{\alpha}
\renewcommand{\b}{\beta}
\newcommand{\normeq}{\trianglelefteqslant}

\newcommand{\e}{\epsilon}
 \renewcommand{\L}{\Lambda}
\renewcommand{\l}{\lambda} \renewcommand{\O}{\Omega}

 \renewcommand{\to}{\rightarrow}

 \newcommand{\C}{\mathcal{C}}

\newcommand{\la}{\langle}
\newcommand{\ra}{\rangle}

\newcommand{\leqs}{\leqslant}
\newcommand{\geqs}{\geqslant}
 
\newcommand{\what}{\widehat} 
 \newcommand{\vs}{\vspace{3mm}}

\makeatletter
\newcommand{\imod}[1]{\allowbreak\mkern4mu({\operator@font mod}\,\,#1)}
\makeatother

\newtheorem{theorem}{Theorem} 
\newtheorem*{theorem*}{Theorem} 
\newtheorem*{conj*}{Conjecture}

\newtheorem{corol}[theorem]{Corollary}

\newtheorem{thm}{Theorem}[section] 
\newtheorem{prop}[thm]{Proposition} 
\newtheorem{lem}[thm]{Lemma}

\theoremstyle{definition}
\newtheorem{rem}[thm]{Remark}
\newtheorem{remk}{Remark}

\newtheorem*{deff}{Definition}
\newtheorem{defn}[thm]{Definition}
\newtheorem{ex}[thm]{Example}

\begin{document}

\author{Timothy C. Burness}
\address{T.C. Burness, School of Mathematics, University of Bristol, Bristol BS8 1UG, UK}
\email{t.burness@bristol.ac.uk}

\author{Lei Wang}
\address{L. Wang, School of Mathematics and Statistics, Yunnan University, Kunming 650091, Yunnan, P.R. China}
\email{wanglei@ynu.edu.cn}

\title[Regularity of irreducible subgroups]{On the regularity of irreducible subgroups \\ of finite classical groups}

\begin{abstract}
Let $G$ be a finite group and let $\tau = (H_1, \ldots, H_t)$ be a $t$-tuple of core-free subgroups of $G$. We say that $\tau$ is regular if $G$ contains elements $g_1, \ldots, g_t$ such that $\bigcap_i H_i^{g_i} = 1$, which is equivalent to the existence of a regular $G$-orbit on the Cartesian product $G/H_1 \times \cdots \times G/H_t$. Regular tuples were first investigated by Anagnostopoulou-Merkouri and Burness in a paper from 2024, partly motivated by the aim of seeking a natural generalisation of the classical and widely studied concept of a base for a transitive permutation group, which aligns with the special case where the $H_i$ are pairwise conjugate subgroups. In this paper, we focus on the case where $G$ is a finite almost simple classical group and each $H_i$ is a maximal subgroup contained in Aschbacher's collection $\mathcal{S}$ of irreducibly embedded subgroups. Our main theorem determines all the non-regular $t$-tuples of this form with $t \geqs 2$, which extends earlier work by Burness, Guralnick and Saxl in the base size setting. In particular, we deduce that every pair of maximal subgroups in $\mathcal{S}$ is regular if $n \geqs 15$, where $n$ is the dimension of the natural module for the socle of $G$, and this lower bound is best possible.
\end{abstract}

\date{\today}

\maketitle

\section{Introduction}\label{s:intro}

Let $G$ be a finite group, let $t \geqs 2$ be an integer and let $H_1, \ldots, H_t$ be a collection of core-free subgroups of $G$, allowing repetitions. Consider the natural action of $G$ on the Cartesian product 
\[
\Gamma = G/H_1 \times \cdots \times G/H_t
\]
and observe that $G$ has a regular orbit on $\Gamma$ if and only if  
\[
\bigcap_{i=1}^t H_i^{g_i} = 1
\]
for some elements $g_i \in G$. Following \cite{AB}, we make the following definition.

\begin{deff}
A $t$-tuple $\tau = (H_1, \ldots, H_t)$ of core-free subgroups of $G$ is \emph{regular} if $G$ has a regular orbit on $\Gamma$, and \emph{non-regular} otherwise. 
\end{deff}

For $t = 2$ or $3$, we will refer to regular and non-regular pairs or triples, respectively. Clearly, the regularity of a given tuple does not depend on the ordering of the subgroups in the tuple. In addition, $\tau = (H_1, \ldots, H_t)$ is regular if and only if $(H_1^{g_1}, \ldots, H_t^{g_t})$ is regular for some $g_i \in G$, so we are free to replace each $H_i$ by any conjugate. It will be convenient to say that $\tau$ is a \emph{conjugate} tuple if each $H_i$ is conjugate to $H_1$. 

Recall that if $G \leqs {\rm Sym}(\O)$ is a transitive permutation group on a finite set $\O$ with point stabiliser $H = G_{\a}$, then a subset of $\O$ is a \emph{base} for $G$ if its pointwise stabiliser in $G$ is trivial. We write $b(G,H)$ for the minimal size of a base, which is called the \emph{base size} of $G$. Determining the base size of a given group is a fundamental problem in permutation group theory, and this classical invariant has been widely studied for many decades, finding an extensive range of applications and connections to other areas (for example, see the survey articles \cite{BC} and \cite[Section 5]{B180}, and the references therein). The connection with subgroup regularity is transparent since we have $b(G,H) \leqs t$ if and only if $G$ has a regular orbit on $(G/H)^t$. In other words, $b(G,H) \leqs t$ if and only if the conjugate $t$-tuple $\tau = (H, \ldots, H)$ is regular. 

Given the extensive literature on bases for finite permutation groups, it is natural to study the possibility of extending some of the main results to the more general regularity setting. Indeed, this was one of the main motivations for introducing the concept of subgroup regularity in \cite{AB}, where several results in this direction are established for almost simple groups.

In order to set the scene for the main results of this paper, first recall that a finite group $G$ is \emph{almost simple} if there exists a nonabelian simple group $G_0$ such that
\[
G_0 \normeq G \leqs {\rm Aut}(G_0).
\]
Here $G_0$ is the unique minimal normal subgroup of $G$, so it coincides with the socle of $G$ (and we may identify $G_0$ with its group of inner automorphisms). Now suppose $G \leqs {\rm Sym}(\O)$ is an almost simple primitive permutation group with point stabiliser $H$, so $H$ is a core-free maximal subgroup of $G$. Roughly speaking, we say that $G$ is \emph{non-standard} if one of the following holds:
\begin{itemize}\addtolength{\itemsep}{0.2\baselineskip}
\item[{\rm (a)}] $G_0 = A_n$ is an alternating group and $H \cap G_0$ acts primitively on $\{1, \ldots, n\}$;
\item[{\rm (b)}] $G_0$ is a classical group and $H \cap G_0$ acts irreducibly on the natural module for $G_0$;
\item[{\rm (c)}] $G_0$ is an exceptional group of Lie type or a sporadic simple group.
\end{itemize}
Otherwise, we say that $G$ is \emph{standard}, and we refer the reader to \cite[Definition 1.1]{B07} for the formal definition. 

It is easy to see that the base size of a standard group can be arbitrarily large. For example, we have $b(G,H) = n-1$ for the natural action of $G = S_n$ on $\O = \{1, \ldots, n\}$. However, 
a well-known conjecture of Cameron from the 1990s (see \cite[p.122]{Cam}) asserts that $b(G,H) \leqs 7$ for every non-standard group, with equality if and only if $G$ is the Mathieu group ${\rm M}_{24}$ in its natural action on $24$ points. A weaker  version of Cameron's conjecture was proved by Liebeck and Shalev in \cite{LSh2}, where they established an upper bound $b(G,H) \leqs c$ for some undetermined absolute constant $c$. The precise form of the conjecture was established in later work by Burness et al. \cite{B07,BGS0,BLS,BOW}, working extensively with a powerful probabilistic approach for bounding the base size of a finite permutation group, which was originally introduced in \cite{LSh2}.

Given an almost simple group $G$, let us say that a core-free maximal subgroup $H$ of $G$ is \emph{non-standard} if the primitive action of $G$ on $G/H$ is non-standard as defined above. Then Cameron's conjecture asserts that every conjugate $7$-tuple of non-standard subgroups of $G$ is regular, and every conjugate $6$-tuple $\tau$ is regular, unless $G = {\rm M}_{24}$ and $\tau = (H, \ldots, H)$ with $H = {\rm M}_{23}$. Given the verification of Cameron's conjecture, it is natural to ask whether or not the ``conjugate" condition in the original formulation can be relaxed, which leads to \cite[Conjecture 1(ii)]{AB}:

\begin{conj*}[Anagnostopoulou-Merkouri \& Burness, 2024]
Let $G$ be a finite almost simple group with socle $G_0$. Then every $7$-tuple of non-standard subgroups of $G$ is regular.
\end{conj*}

In \cite{AB}, a strong version of this conjecture is proved in the special cases where $G_0$ is an alternating or a sporadic group (see \cite[Theorem 1(ii) and Theorem 2(i)]{AB}). And the problem for groups of Lie type is addressed in \cite{A}, where  Anagnostopoulou-Merkouri completes the proof of the conjecture. In particular, she shows that every $6$-tuple $\tau$ of non-standard subgroups is regular, unless $G = {\rm M}_{24}$ and $\tau = (H, \ldots, H)$ with $H = {\rm M}_{23}$. Moreover, if $G$ is a classical group, then \cite{A} shows that every $5$-tuple is regular, which provides a natural generalisation of the main theorem of \cite{B07} on base sizes for non-standard classical groups.

In this paper, we will focus on subgroup regularity for classical groups. In order to define the problem, let $G$ be a finite almost simple classical group over $\mathbb{F}_q$ with socle $G_0$. Write $q = p^f$, where $p$ is a prime, and let $V$ be the natural module for $G_0$. Set $n = \dim V$. In view of the existence of exceptional isomorphisms among certain low-dimensional classical groups (see \cite[Proposition 2.9.1]{KL}), we may (and will) assume throughout this paper that $G_0$ is one of the following:
\begin{equation}\label{e:groups}
\begin{array}{ll}
\mbox{Linear:} \; & \hspace{4.4mm} {\rm L}_n(q), \; \mbox{$n \geqs 2$} \\
\mbox{Unitary:} \; & \hspace{4.4mm} {\rm U}_n(q), \; \mbox{$n \geqs 3$} \\
\mbox{Symplectic:} \; & \hspace{4.4mm} {\rm PSp}_n(q), \; \mbox{$n \geqs 4$ even, $(n,q) \ne (4,2)$} \\
\mbox{Orthogonal:} \; & \left\{ \begin{array}{l}
\O_n(q),\; \mbox{$n \geqs 7$, $nq$ odd} \\
{\rm P\O}_{n}^{\pm}(q),\; \mbox{$n \geqs 8$ even,}
\end{array}\right.
\end{array}
\end{equation}
where we adopt the notation from \cite{KL}. For example, we may assume $(n,q) \ne (4,2)$ when $G_0$ is a symplectic group since there is an isomorphism ${\rm PSp}_4(2)' \cong {\rm L}_2(9)$.

For now, let us exclude the groups with $G_0 = {\rm Sp}_4(q)$ (for $q$ even) or $G_0 = {\rm P\O}_8^{+}(q)$, which require special attention. If $H$ is a core-free maximal subgroup of $G$, then a celebrated theorem of  Aschbacher \cite{asch} implies that $H$ either belongs to one of eight subgroup collections defined in terms of the geometry of $V$ (denoted by $\mathcal{C}_1, \ldots, \mathcal{C}_8$), or $H$ is almost simple and acts irreducibly on $V$. The latter collection of irreducibly embedded almost simple subgroups is denoted by $\mathcal{S}$, and each subgroup in $\mathcal{S}$ satisfies several additional properties (see Definition \ref{sdef}) that are designed to ensure that it is not contained in one of the geometric subgroup collections arising in the statement of Aschbacher's theorem. We refer the reader to Remarks \ref{r:S} and \ref{r:S2} for clarification on the precise definition of $\mathcal{S}$ we adopt in this paper, which includes the special cases where $G_0 = {\rm Sp}_4(q)$ (with $q$ even) or $G_0 = {\rm P\O}_8^{+}(q)$.

There is an extensive literature on the maximal subgroups of classical groups, stretching all the way back to the 19th century. In \cite{KL}, for example, Kleidman and Liebeck give detailed structural information on the subgroups in each $\C_i$ collection and they determine all the geometric maximal subgroups (up to conjugacy) under the assumption $n \geqs 13$. This is complemented by more recent work of Bray, Holt and Roney-Dougal \cite{BHR}, which completely determines all the maximal subgroups (up to conjugacy, and including the subgroups in $\mathcal{S}$) when $n \leqs 12$. The problem of determining the maximal subgroups in $\mathcal{S}$ for all $n \geqs 13$ is intractable, in general, but there are results for $13 \leqs n \leqs 17$ in the PhD theses of Rogers \cite{Rogers} and Schr\"{o}der \cite{Schroder}.
 
By the main theorem of \cite{B07}, if $G \leqs {\rm Sym}(\O)$ is a non-standard classical group with point stabiliser $H$, then $b(G,H) \leqs 5$, with equality if and only if $G = {\rm U}_6(2).2$ and $H = {\rm U}_4(3).2^2$. More refined versions of this result have been obtained in certain special cases. For example, the exact base size is computed in \cite{B21} in the special case where $H$ is soluble, which plays a key role in establishing a conjecture of Vdovin on the base sizes of primitive permutation groups with soluble point stabilisers. And in \cite{BGS}, Burness, Guralnick and Saxl determine the exact base size in the special case where $H$ is contained in Aschbacher's collection $\mathcal{S}$. 

Our goal in this paper is to extend the main theorem of \cite{BGS} by completely determining all the non-regular tuples of subgroups in $\mathcal{S}$. Our main result is the following (note that Tables \ref{tab:3} and \ref{tab:2} are presented at the end of the paper in Section \ref{s:tables}). 

\begin{theorem}\label{t:main}
Let $G$ be a finite almost simple classical group over $\mathbb{F}_q$ with socle $G_0$ and let $\tau$ be a $t$-tuple of maximal subgroups in $\mathcal{S}$ with $t \geqs 2$.
\begin{itemize}\addtolength{\itemsep}{0.2\baselineskip}
\item[{\rm (i)}] If $t \geqs 5$ then $\tau$ is regular.
\item[{\rm (ii)}] If $t=4$, then $\tau$ is non-regular if and only if $G = {\rm U}_6(2).2$ and $\tau = (H,H,H,H)$ with $H = {\rm U}_4(3).2^2$.
\item[{\rm (iii)}] If $t=3$, then $\tau$ is non-regular if and only if one of the following holds:

\vspace{1mm}

\begin{itemize}\addtolength{\itemsep}{0.3\baselineskip}
\item[{\rm (a)}] $G_0 = \O_7(q)$ {\rm (}$q$ odd{\rm )} or ${\rm Sp}_6(q)$ {\rm (}$q$ even{\rm )} and $\tau = (H_1,H_2,H_3)$, where each $H_i$ has socle $G_2(q)$.
\item[{\rm (b)}] $(G,\tau)$ is one of the cases recorded in Table \ref{tab:3}.
\end{itemize}

\vspace{1mm}

\item[{\rm (iv)}] If $t=2$ then $\tau$ is non-regular if and only if one of the following holds:

\vspace{1mm}

\begin{itemize}\addtolength{\itemsep}{0.3\baselineskip}
\item[{\rm (a)}] $G_0 = \O_7(q)$ {\rm (}$q$ odd{\rm )} or ${\rm Sp}_6(q)$ {\rm (}$q$ even{\rm )} and $\tau = (H_1,H_2)$, where each $H_i$ has socle $G_2(q)$.
\item[{\rm (b)}] $G_0 = {\rm Sp}_4(q)$, $q = 2^{f}$ with $f \geqs 3$ odd, and $\tau = (H,H)$, where $H_0 = {}^2B_2(q)$.
\item[{\rm (c)}] $(G,\tau)$ is one of the cases recorded in Table \ref{tab:2}.
\end{itemize}
\end{itemize}
\end{theorem}

\begin{corol}\label{c:main}
If $n \geqs 15$, then every pair of subgroups in $\mathcal{S}$ is regular.
\end{corol}

\begin{remk}\label{r:main}
Some remarks on the statement of Theorem \ref{t:main} are in order.
\begin{itemize}\addtolength{\itemsep}{0.2\baselineskip}
\item[{\rm (a)}] The non-regular $t$-tuples arising in the statement of Theorem \ref{t:main} are described up to conjugacy and re-ordering. We refer the reader to Remark \ref{r:tables} for information on the notation we adopt in Tables \ref{tab:3} and \ref{tab:2}.

\item[{\rm (b)}] It is important to note that our definition of $\mathcal{S}$ excludes a handful of possibilities for $(G,H)$ since the action of $G$ on $G/H$ in these special cases is permutation isomorphic to a standard action. This convention is consistent with the set-up adopted in \cite{BGS} and we refer the reader to Remarks \ref{r:S} and \ref{r:S2} for more details.

\item[{\rm (c)}] For $q$ odd in parts (iii)(a) and (iv)(a), we note that $G_0 = \O_7(q)$ and $G$ has two conjugacy classes of maximal subgroups in $\mathcal{S}$ with socle $G_2(q)$ (see \cite[Table 8.40]{BHR}).

\item[{\rm (d)}] The existence of the non-regular pair $\tau = (S_{16},S_{16})$ for $G = {\rm O}_{14}^{+}(2)$ shows that the lower bound on $n$ in Corollary \ref{c:main} is optimal.

\item[{\rm (e)}] In the proof of Theorem \ref{t:main}, we identified an error in the statement of \cite[Theorem 1]{BGS}, which we take the opportunity to correct here. If $G = {\rm L}_3(4).2^2$ and $H = A_6.2^2$ then $b(G,H) = 4$ and thus $(H,H,H)$ is a non-regular triple (in \cite[Table 1]{BGS}, it is incorrectly stated that $b(G,H) = 3$).

\item[{\rm (f)}] Note that Theorem \ref{t:main} extends a special case of \cite[Theorem 3(iii)]{A}, which implies that every $5$-tuple of subgroups in $\mathcal{S}$ is regular.
\end{itemize}
\end{remk}

As commented above, probabilistic methods based on fixed point ratio estimates played a central role in the proof of Cameron's base size conjecture in \cite{B07,BGS0,BLS,BOW}, as well as some of the subsequent refinements in \cite{B18, B21,BGS}. As observed in \cite{AB}, a suitable modification of the probabilistic approach also applies in the more general regularity setting. Given a tuple $\tau = (H_1, \ldots, H_t)$  of core-free subgroups of $G$, the goal is to estimate the probability $Q(G,\tau)$ that a uniformly random element in the Cartesian product $\Gamma = G/H_1 \times \cdots \times G/H_t$  is not contained in a regular $G$-orbit, so $\tau$ is regular if and only if $Q(G,\tau)<1$. In order to estimate this probability, we have \cite[Lemma 2.1]{AB}, which states that
\[
Q(G,\tau) \leqs \sum_{i=1}^{\ell} |x_i^G| \cdot \left(\prod_{j=1}^t {\rm fpr}(x_i,G/H_j)\right) =: \what{Q}(G,\tau),
\]
where $x_1, \ldots, x_{\ell}$ is a complete set of representatives of the conjugacy classes in $G$ of elements of prime order, and 
\[
{\rm fpr}(x,G/H) = \frac{|x^G \cap H|}{|x^G|}
\]
is the \emph{fixed point ratio} of $x \in G$ with respect to the natural transitive action of $G$ on $\O = G/H$ (this is simply the proportion of points in $\O$ fixed by $x$).

Let us consider a pair of core-free subgroups $\tau = (H_1,H_2)$. Firstly, it is worth noting that there are examples where $b(G,H_1) = b(G,H_2) = 2$ and $\tau$ is non-regular (so in such a situation, there exist elements $x_i \in G$ such that $H_i \cap H_i^{x_i} = 1$ for $i=1,2$, but $H_1 \cap H_2^y \ne 1$ for all $y \in G$). Similarly, there are examples where $\tau$ is regular and $b(G,H_i) \geqs 3$ for $i=1,2$; we refer the reader to Remark \ref{r:key} for more details. However, if we set $\what{Q}(G,H_i) = \what{Q}(G,\a_i)$ for $\a_i = (H_i,H_i)$, then it is straightforward to show that
\[
\what{Q}(G,\tau) \leqs \frac{1}{2}\left(\what{Q}(G,H_1) + \what{Q}(G,H_2)\right),
\]
which means that $\tau$ is regular if $\what{Q}(G,H_i)<1$ for $i=1,2$ (see Lemma \ref{l:key} for a more general statement). 

This key observation allows us to exploit the proof of the main theorem of \cite{BGS}, further underlining the power and applicability of the probabilistic approach for studying base sizes. Indeed, by suitably extending the proof of \cite[Theorem 1]{BGS}, we can establish Theorem \ref{t:hat} below, which effectively reduces the proof of Theorem \ref{t:main} to the low-dimensional classical groups appearing in \cite[Table 1]{BGS} and Table \ref{tab:new}. Fortunately, most of the remaining cases are amenable to direct computation using {\sc Magma} \cite{magma}, allowing us to work with some of the computational techniques for studying subgroup regularity introduced in \cite{AB_comp}, which played a key role in \cite{AB}. We will say more about our application of computational methods in Section \ref{ss:comp}. Recall that $\what{Q}(G,H) = \what{Q}(G,\a)$, where $\a = (H,H)$.

\begin{theorem}\label{t:hat}
Let $G$ be a finite almost simple classical group with socle $G_0$ and let $H \in \mathcal{S}$ be a maximal subgroup of $G$. Then $\widehat{Q}(G,H) \geqs 1$ if and only if one of the following holds:
\begin{itemize}\addtolength{\itemsep}{0.2\baselineskip}
\item[(i)] $b(G,H) \geqs 3$, in which case $(G,H)$ is recorded in \cite[Table 1]{BGS}.
\item[(ii)] $b(G,H) = 2$ and $(G,H)$ is recorded in Table \ref{tab:new}.
\end{itemize}
\end{theorem}

\renewcommand{\arraystretch}{1.1}

\begin{table} 
$$\begin{array}{llc} \hline
G & H & \what{Q}(G,H) \\ \hline 
\O_{14}^{+}(2) & A_{16} & 6.10 \\
{\rm Sp}_{12}(2) & S_{14} & 8.32 \\ 
{\rm P\O}_{8}^{+}(5).S_3 & \O_8^{+}(2).S_3 & 2.20 \\
{\rm U}_5(2).2 & {\rm PGL}_2(11) & 1.62 \\
{\rm U}_5(2) & {\rm L}_2(11) & 1.39 \\
{\rm U}_4(5).2_1 & {\rm U}_4(2).2 & 2.10 \\
{\rm PGSp}_4(7) & S_7 & 1.49 \\
{\rm PGSp}_4(5) & S_6 & 1.29 \\
{\rm L}_2(31) & A_5 & 1.18 \\ 
{\rm L}_2(29) & A_5 & 1.36 \\\hline
\end{array}$$
\caption{The pairs $(G,H)$ in Theorem \ref{t:hat}(ii)}
\label{tab:new}
\end{table}

\renewcommand{\arraystretch}{1}

Let us observe that the remarkably short list of special cases recorded in Table \ref{tab:new} coincides with the complete collection of base-two groups in this setting for which the probabilistic method, via fixed point ratio estimates, is not able to detect the base-two property. 

\begin{remk}\label{r:2}
In the final column of Table \ref{tab:new}, we record $\what{Q}(G,H)$ rounded to $3$ significant figures. In addition, we adopt the standard notation for the group $G = {\rm U}_4(5).2_1$ (this is consistent with \cite{GAPCTL}, for example), which indicates that $G$ contains an involutory graph automorphism $x$ with $C_{G_0}(x) = {\rm PGSp}_4(5)$.
\end{remk}

\begin{remk}
Once again, we remind the reader that our definition of $\mathcal{S}$ in Theorem \ref{t:hat} excludes a handful of pairs $(G,H)$ for which the action of $G$ on $G/H$ is permutation isomorphic to a standard action (see Remarks \ref{r:S} and \ref{r:S2}). 
\end{remk}

\vs

\noindent \textbf{Notation.} Our group-theoretic notation is all fairly standard. As noted above, we use the notation from \cite{KL} for simple groups of Lie type. And if $G$ is a finite group and $r$ is a positive integer, then we write $i_r(G)$ for the number of elements in $G$ of order $r$. Given a pair of positive integers $a,b$, we use the familiar Kronecker delta notation $\delta_{a,b}$, which means that $\delta_{a,b} = 1$ if $a=b$, otherwise $\delta_{a,b} = 0$. We also write $(a,b)$ for the greatest common divisor of $a$ and $b$, and if $q$ is a prime power then we set $Q = q/(q+1)$.

\vs

\noindent \textbf{Organisation.} In Section \ref{s:prel} we present a number of preliminary results, which will be needed in the proofs of Theorems \ref{t:main} and \ref{t:hat}. For example, in Section \ref{ss:class} we give a formal definition of Aschbacher's collection $\mathcal{S}$ (see Definition \ref{sdef}) and we recall a key result of Guralnick and Saxl (see Theorem \ref{gursax}), which imposes strong restrictions on the elements contained in a maximal subgroup in 
$\mathcal{S}$. In Sections \ref{ss:prob} and \ref{ss:comp}, we discuss the main probabilistic and computational methods we will use to prove our main theorems. We then focus on the proof of Theorem \ref{t:hat} in Section \ref{s:hat}; by working closely with the proof of \cite[Theorem 1]{BGS}, the main part of the argument requires a detailed analysis of the subgroups in $\mathcal{S}$ for the low-dimensional classical groups with $n \leqs 14$, which allows us to work closely with \cite{BHR,Schroder}. The proof of Theorem \ref{t:main} is presented in Section \ref{s:main}, which combines Theorem \ref{t:hat} and \cite[Theorem 1]{BGS} to reduce the proof to a small number of concrete cases, many of which are amenable to direct computation, using the methods discussed in Section \ref{ss:comp}. Finally, the two main tables referred to in the statement of Theorem \ref{t:main} are presented in Section \ref{s:tables}.

\vs

\noindent \textbf{Acknowledgements.} We thank  Marina Anagnostopoulou-Merkouri and an anonymous referee for their careful reading of an earlier version of the paper and for their very helpful corrections, comments and suggestions. Lei Wang thanks the China Scholarship Council for their financial support, and he thanks the School of Mathematics at the University of Bristol for hosting a research visit in 2026. He also acknowledges the support of NSFC Grants 12571022 and 12061083.

\section{Preliminaries}\label{s:prel}

In this section, we present several preliminary results that we will need in the proofs of Theorems \ref{t:main} and \ref{t:hat}. Let $G$ be an almost simple classical group with natural module $V$ and let $H \in \mathcal{S}$ be a maximal subgroup of $G$.

In Section \ref{ss:class}, we begin by recalling the definition of Aschbacher's collection $\mathcal{S}$ of maximal subgroups from \cite{asch}. We also present a deep theorem of Guralnick and Saxl concerning the action of elements in $H$ on $V$ (see Theorem \ref{gursax}), which we can use to derive lower bounds on $|x^G|$ for each element $x \in H$ of prime order. The latter bounds are needed in order to effectively apply our probabilistic approach for determining the regularity status of a given tuple of core-free subgroups, as explained in Section \ref{ss:prob}. Here we also present Lemma \ref{l:key}, which is a new observation that plays a key role in the proof of Theorem \ref{t:main}, and it explains why we are interested in establishing Theorem \ref{t:hat}. Finally, in Section \ref{ss:comp} we outline some of the main computational methods we use in this paper.

\subsection{Classical groups and the collection $\mathcal{S}$}\label{ss:class}

Let $G$ be an almost simple classical group over $\mathbb{F}_{q}$ with socle $G_0$ and natural module $V$. Set $n = \dim V$ and write $q = p^f$, where $p$ is a prime. As pointed out in Section \ref{s:intro}, we may (and will) assume throughout that $G_0$ is one of the groups recorded in \eqref{e:groups}.

Assume for now that $G_0 \not\in \mathcal{G}$, where 
\[
\mathcal{G} = \{ {\rm Sp}_4(q) \, \mbox{(with $q$ even)},\; {\rm P\O}_8^{+}(q) \},
\]
and let $H$ be a core-free maximal subgroup of $G$. By a well-known theorem of Aschbacher \cite{asch}, either $H$ belongs to one of eight subgroup collections $\C_1, \ldots, \C_8$ that are defined in terms of the geometry of $V$ (for example, the subgroups comprising $\C_1$ are the stabilisers of certain subspaces, or pairs of subspaces, of $V$), or $H$ is almost simple and acts irreducibly on $V$. The latter collection is denoted by $\mathcal{S}$ and following \cite[p.3]{KL} we give the following formal definition. Notice that the various conditions arising in this definition are designed to ensure that a subgroup in $\mathcal{S}$ is not contained in a geometric subgroup.

\begin{defn}\label{sdef}
Assume $G_0 \not\in \mathcal{G}$. Then a subgroup $H$ of $G$ belongs to the collection $\mathcal{S}$ if and only if it satisfies the following conditions:
\begin{itemize}\addtolength{\itemsep}{0.2\baselineskip}
\item[(i)] The socle $H_0$ of $H$ is a nonabelian simple group and $H_0 \not\cong G_0$.
\item[(ii)] If $\widehat{H}_0$ is the full covering group of $H_0$, and if $\rho : \widehat{H}_0 \to {\rm GL}(V)$ is a representation of $\widehat{H}_0$ such that $\rho(\widehat{H}_0)=H_0$ (modulo scalars), then $\rho$ is absolutely irreducible.
\item[(iii)] $\rho(\widehat{H}_0)$ cannot be realised over a proper subfield of $\mathbb{F}$, where 
$\mathbb{F} = \mathbb{F}_{q^2}$ if $G_0={\rm U}_{n}(q)$ is a unitary group, otherwise $\mathbb{F} = \mathbb{F}_{q}$.
\item[(iv)] If $\rho(\widehat{H}_0)$ fixes a nondegenerate quadratic form on $V$, then $G_0={\rm P\Omega}_{n}^{\e}(q)$.
\item[(v)] If $\rho(\widehat{H}_0)$ fixes a nondegenerate alternating form on $V$, but no nondegenerate quadratic form, then $G_0={\rm PSp}_{n}(q)$.
\item[(vi)] If $\rho(\widehat{H}_0)$ fixes a nondegenerate hermitian form on $V$, then $G_0={\rm U}_{n}(q)$.
\item[(vii)] If $\rho(\widehat{H}_0)$ does not fix a form as in (iv), (v) or (vi), then $G_0={\rm L}_{n}(q)$.
\end{itemize}
\end{defn}

As we highlighted in Section \ref{s:intro}, the maximal subgroups in $\mathcal{S}$ are determined in \cite{BHR} (up to conjugacy in $G$) for all of the classical groups with $n \leqs 12$. This classification has been extended to the groups with $n \in \{13,14,15\}$ by Schr\"{o}der \cite{Schroder}, with an almost complete result for $n \in \{16,17\}$ determined by Rogers in \cite{Rogers}. It remains an open problem to completely determine the subgroups in $\mathcal{S}$ for $n \geqs 18$, which is an intractable problem, in general.

\begin{rem}\label{r:S}
There is a version of Aschbacher's theorem for the groups with $G_0 \in \mathcal{G}$. However, the precise definition of the relevant subgroup collections is complicated by the existence of exceptional graph automorphisms that arise in these special cases. Indeed, a version for $G_0 = {\rm Sp}_4(q)$ (with $q$ even) is presented in \cite[Section 14]{asch}, while the groups with socle $G_0 = {\rm P\O}_8^{+}(q)$ were handled in a later paper by Kleidman \cite{K}. 
Here we explain how this impacts our definition of the collection $\mathcal{S}$ in these two special cases. 

\begin{itemize}\addtolength{\itemsep}{0.2\baselineskip}
\item[{\rm (i)}] Following \cite[Table 8.14]{BHR}, if $G_0 = {\rm Sp}_4(q)$ with $q = 2^f$, then $\mathcal{S}$ is non-empty if and only if $f \geqs 3$ is odd, in which case $\mathcal{S}$ comprises a unique conjugacy class of subgroups of the form $H = N_G(L)$, where $L = {}^2B_2(q)$.

\item[{\rm (ii)}] For $G_0 = {\rm P\O}_8^{+}(q)$ we refer to \cite[Table 8.50]{BHR} and we say that a maximal subgroup $H$ of $G$ is in $\mathcal{S}$ if it is of the form $H = N_G(L)$, where $L$ is one of the following:
\[
\hspace{15mm} {\rm L}_3^{\e}(q) \, (q \equiv \e \imod{3}), \, {}^3D_4(q_0) \, (q = q_0^3), \, \O_8^{+}(2) \, (q=p \geqs 3),
 \]
 \[
\hspace{15mm}  {}^2B_2(8) \, (q = 5), A_{10} \, (q=5), \, A_9 \, (q=2). 
\]
In particular, we exclude the following possibilities for $L$:
\begin{equation}\label{e:exc}
\hspace{15mm} \O_7(q) \, (p \geqs 3), \, {\rm Sp}_6(q) \, (p=2), \, {\rm P\O}_8^{-}(q_0) \, (q = q_0^2).
\end{equation}
This is justified because in each of the latter cases, the action of $G$ on $G/H$ is permutation isomorphic to the action of $G$ on $G/K$, where $K$ is a maximal subgroup in one of the geometric subgroup collections. For example, if $q$ is odd and $H = N_G(L)$ for some irreducibly embedded subgroup $L = \O_7(q)$ of $G_0$, then the action of $G$ on $G/H$ is permutation isomorphic to the action of $G$ on $G/K$, where $K \in \mathcal{C}_1$ is the stabiliser of a nondegenerate $1$-dimensional subspace of $V$.

\item[{\rm (iii)}] It is worth noting that the subgroups in $\mathcal{S}$ are not necessarily almost simple in these special cases. For example, if $G_0 = {\rm Sp}_4(q)$ and $G = G_0.\la \gamma \ra$, where $q = 2^f$ with $f \geqs 3$ odd and $\gamma$ is an involutory graph automorphism, then each $H \in \mathcal{S}$ is of the form $H = {}^2B_2(q) \times 2$. Similarly, if $G_0 = {\rm P\O}_8^{+}(q)$, $p \ne 3$ and $G = G_0.\la \theta \ra$, where $\theta$ is a triality graph automorphism, then there are subgroups in $\mathcal{S}$ of the form $H = {\rm PGL}_3^{\e}(q) \times 3$, where $q \equiv \e \imod{3}$. 
\end{itemize}
\end{rem}

\begin{rem}\label{r:S2}
In addition to the subgroups in \eqref{e:exc} that we exclude when defining the collection $\mathcal{S}$ for $G_0 = {\rm P\O}_8^{+}(q)$, we make two further exclusions, as detailed below: 
\begin{itemize}\addtolength{\itemsep}{0.2\baselineskip}
\item[{\rm (i)}] If $G_0 = {\rm L}_4(2)$ then \cite[Table 8.9]{BHR} records a subgroup $H \in \mathcal{S}$ with socle $A_7$. Here  $G_0 \cong A_8$ and the action of $G$ on $G/H$ is permutation isomorphic to the natural action of $S_8$ or $A_8$ on $\{1, \ldots, 8\}$. Therefore, we exclude $(G_0,H_0) = ({\rm L}_4(2),A_7)$ since we are only interested in non-standard groups (in order to be consistent with \cite{BGS}, for example).
\item[{\rm (ii)}] Similarly, if $G_0 = {\rm L}_2(9)$ then \cite[Table 8.2]{BHR} lists a subgroup $H \in \mathcal{S}$ with socle $A_5$. We exclude this case because $G_0 \cong A_6$ and the action of $G$ on $G/H$ is permutation isomorphic to the standard action of $S_6$ or $A_6$ on $\{1, \ldots, 6\}$.
\end{itemize}
Notice that by excluding the pairs $(G_0,H_0) = ({\rm L}_4(2),A_7), ({\rm L}_2(9),A_5)$, together with the cases listed in \eqref{e:exc} for $G_0 = {\rm P\O}_8^{+}(q)$, our definition of the collection $\mathcal{S}$ agrees with the definition adopted in \cite{BGS} (see \cite[Table 2]{BGS}, noting that the case listed in the third row of this table, with $G_0 = {\rm L}_6^{\e}(q)$, does not arise).
\end{rem}

Next we turn to the elements contained in a maximal subgroup $H \in \mathcal{S}$. Here we are mainly interested in prime order elements $x \in H \cap {\rm PGL}(V)$ and we seek restrictions on the action of $x$ on $V$, which we can then use to determine a lower bound on $|x^G|$. Here the main result is a theorem of Guralnick and Saxl, stated as Theorem \ref{gursax} below, which requires us to introduce some additional notation and terminology.  

Suppose $x \in G \cap {\rm PGL}(V)$ is nontrivial and write $x = \hat{x}Z$, where $\hat{x} \in {\rm GL}(V)$ and $Z$ is the centre of ${\rm GL}(V)$. Set $\bar{V} = V \otimes k$, where $k$ is the algebraic closure of $\mathbb{F}_q$, and define 
\[
\nu(x) = \min\{\dim [\bar{V},\l\hat{x}] \,:\, \l \in k^{\times}\}
\]
where $[\bar{V},\l\hat{x}]$ is the subspace $\la v-v^{\l\hat{x}} \,:\, v \in \bar{V} \ra$. 
Notice that $\nu(x)$ is the codimension of the largest eigenspace of $\hat{x}$ on $\bar{V}$, and this does not depend on the choice of coset representative $\hat{x}$. In addition, if $H$ is a subgroup of $G$ then it will be convenient to set 
\begin{equation}\label{e:nuH}
\nu(H) = \min\{ \nu(x) \,:\, \mbox{$x \in H \cap {\rm PGL}(V)$ has prime order} \}.
\end{equation}

In order to state Theorem \ref{gursax} below, we need to define two important subcollections of $\mathcal{S}$, which are usually denoted by the symbols $\mathcal{A}$ and $\mathcal{C}$.

To define the collection $\mathcal{A}$, let $H$ be an almost simple group with socle $H_0 = A_m$ and let $p$ be a prime. Consider the natural action of $S_m$ on the permutation module $\mathbb{F}_{p}^m$ and define the following subspaces  
\[
U=\{(a_1, \ldots, a_m) \,:\, \sum_{i=1}^{m}a_i=0\},\;\; W=\{(a,\ldots, a) \,:\, a \in \mathbb{F}_{p}\}.
\]
Notice that $U$ and $W$ are the only nonzero proper $H_0$-invariant submodules of $\mathbb{F}_{p}^m$. Then
$X=U/(U \cap W)$ is an absolutely irreducible $H_0$-module over $\mathbb{F}_{p}$, which we refer to as the \emph{fully deleted permutation module} for $H_0$. Note that if $d = \dim X$ then 
$d = m-2$ if $p$ divides $m$, otherwise $d = m-1$. If $m \geqs 10$ then $X$ has the smallest dimension of all nontrivial irreducible $H_0$-modules over $\mathbb{F}_{p}$ (see \cite[Proposition 5.3.5]{KL}). Let us also observe that $H_0$ preserves the symmetric bilinear form $\b:U \times U \to \mathbb{F}_{p}$ defined by
\[
\b((a_1, \ldots, a_m),(b_1, \ldots, b_m)) = \sum_{i=1}^{m}a_ib_i,
\]
which induces a symmetric bilinear form on $X$. 

In this way, the action of $H_0$ on $X$ yields an embedding of $H_0$ in an appropriate symplectic or orthogonal group defined over $\mathbb{F}_p$. So if $G$ is an almost simple symplectic or orthogonal group over $\mathbb{F}_p$, then the collection $\mathcal{A}$ comprises the subgroups $H$ whose socle $H_0 = A_m$ embeds in $G_0$ via the fully deleted permutation module for $H_0$, which we identify with the natural module $V$ for $G_0$. The possibilities that arise are recorded in Table \ref{atab}. In particular, note that the collection $\mathcal{A}$ is empty if $G$ is a linear or unitary group, or if $q \ne p$.

\renewcommand{\arraystretch}{1.1}
\begin{table}
$$\begin{array}{lll} \hline
G_0 & m & \mbox{Conditions} \\ \hline
{\rm Sp}_n(2) & n+2 & n \geqs 8,\, n \equiv 0 \imod{4} \\
\O_{n}^{+}(2) & n+2 & n \geqs 14,\, n \equiv 6 \imod{8} \\
& n+1 & n \geqs 8,\, n \equiv 0 \imod{8} \\
\O_{n}^{-}(2) & n+2 & n \geqs 10,\, n \equiv 2 \imod{8} \\
 & n+1 & n \geqs 12,\, n \equiv 4 \imod{8} \\
{\rm P\Omega}_n^{\e}(p) & n+2 & n \geqs 7, \, p \geqs 3, \, \mbox{$p$ divides $n+2$} \\
& n+1 & n \geqs 7, \, p \geqs 3, \, (p,n+2) = (p,n+1) = 1 \\ \hline
\end{array}$$
\caption{The collection $\mathcal{A}$, $H_0=A_{m}$}
\label{atab}
\end{table}
\renewcommand{\arraystretch}{1}

\begin{rem}\label{r:max0}
Notice that in Table \ref{atab} we have excluded the cases where $H_0 = A_{n+1}$ and $G_0 = \O_n^{\e}(2)$ with $n \equiv 2\e \imod{4}$. This is due to the fact that $H_0 < A_{n+2} < G_0$ in both cases, so $H$ is non-maximal.
\end{rem}

The collection $\mathcal{C}$ is defined in Table \ref{ctab}, where $H_0$ denotes the socle of $H$. In each case, we have $n \leqs 10$ and so the precise structure of $H$, as well as the exact  conditions on $q$ required for the maximality of $H$, can be read off from the relevant tables in \cite[Chapter 8]{BHR}. 

\renewcommand{\arraystretch}{1.1}
\begin{table} 
$$\begin{array}{lll} \hline
G_0 & H_0 & \mbox{Conditions} \\ \hline
{\rm PSp}_{10}(q) & {\rm U}_{5}(2) &  q = p \geqs 3 \\ 
{\rm P\Omega}_{8}^{+}(q) & {}^{3}D_{4}(q_{0}) & q=q_{0}^{3} \\
 & \Omega_{8}^{+}(2) & q = p \geqs 3 \\
{\rm L}_{7}^{\e}(q) & {\rm U}_{3}(3) & q = p\equiv \e\imod{4}, \, p \geqs 5 \\
\Omega_{7}(q) & G_2(q) & p \geqs 3 \\
& {\rm Sp}_{6}(2) & q = p \geqs 3 \\
{\rm L}_{6}^{\e}(q) &  {\rm L}_{3}^{\e}(q) &  p \geqs 3 \\
&   {\rm U}_{4}(3) &  q = p\equiv \e\imod{6} \mbox{ or } (\e,q) = (-,2) \\
&  {\rm L}_{3}(4) & q = p\equiv \e\imod{6} \\
& A_{7} & q \in \{p,p^2\}, \, p \geqs 5 \\
&  A_{6} & q \in \{p,p^2\}, \, p \geqs 5 \\
&  {\rm M}_{22} & (\e,q) = (-,2) \\
&  {\rm M}_{12} & (\e,q) = (+,3) \\
{\rm PSp}_{6}(q) &  G_2(q) & p = 2, \, q \geqs 4 \\
&  {\rm J}_{2} & q \in \{p,p^2\}, \, p \geqs 3  \\
&  {\rm U}_{3}(3) & q = p \geqs 7 \mbox{ or } q = 2 \\ \hline
\end{array}$$
\caption{The collection $\mathcal{C}$}
\label{ctab}
\end{table}
\renewcommand{\arraystretch}{1}

We are now ready to present the main result on $\nu(H)$ for subgroups $H \in \mathcal{S}$. In the statement, we continue (as always) to assume that $G_0$ is one of the groups defined in \eqref{e:groups} and we work (as always) with the definition of $\mathcal{S}$ discussed above (see Remarks \ref{r:S} and \ref{r:S2}).

\begin{thm}\label{gursax}
Let $G$ be a finite almost simple classical group with socle $G_0$ and let $H \in \mathcal{S}$ be a maximal subgroup of $G$. If $n \geqs 6$, where $n = \dim V$ is the dimension of the natural module for $G_0$, then either $H \in \mathcal{A} \cup \mathcal{C}$ or $\nu(H) > \max\{2,\sqrt{n}/2\}$.
\end{thm}

\begin{proof}
This is a special case of \cite[Theorem 7.1]{GS}.
\end{proof}

Detailed information on the conjugacy classes of prime order elements in the almost simple classical groups is presented in \cite[Chapter 3]{BG} and we will freely refer to the notation and terminology therein. In particular, if $q$ is even and $G_0$ is a symplectic or orthogonal group, then we will frequently refer to unipotent involutions of type $a_i$, $b_i$ and $c_i$ (see \cite[Sections 3.4.4, 3.5.4]{BG}), which is consistent with the original notation introduced by Aschbacher and Seitz in \cite{AS}.

Given such an element $x \in G \cap {\rm PGL}(V)$, bounds on $|x^G|$ in terms of $\nu(x)$ are derived in \cite[Section 3]{FPR2}. In this way, if $n \geqs 6$ and $H \in \mathcal{S} \setminus (\mathcal{A} \cup \mathcal{C})$, then we can often use the lower bound on $\nu(H)$ in Theorem \ref{gursax} to derive a lower bound on $|x^G|$, which is valid for all prime order elements $x \in H \cap {\rm PGL}(V)$. With this goal in mind, we present the following proposition, which will be useful in the proof of Theorem \ref{t:hat}. In the statement of the proposition, and for the remainder of the paper, we set $Q = q/(q+1)$.

\begin{prop}\label{p:xgbd}
Suppose $G_0 = {\rm L}_n^{\e}(q)$ $(n \geqs 9)$, ${\rm PSp}_n(q)$ $(n \geqs 6)$, or ${\rm P\O}_n^{\e}(q)$ $(n \geqs 7)$ and let $H$ be a subgroup of $G$ with $\nu(H) \geqs 3$. Then 
\[
|x^G|>\frac{1}{a}Qq^{3b(n-3-c)}
\]
for all $x \in H$ of prime order, where
\[
(a,b,c) = \left\{ \begin{array}{ll}
(2,2,0) & \mbox{if $G_0 = {\rm L}_n^{\e}(q)$} \\
(4,1,0) & \mbox{if $G_0 = {\rm PSp}_n(q)$} \\
(4,1,1) & \mbox{if $G_0 = {\rm P\O}_n^{\e}(q)$.}
\end{array}\right.
\] 
\end{prop}

\begin{proof}
Suppose $x \in H$ has prime order and first assume $G_0 = {\rm L}_n^{\e}(q)$ with $n \geqs 9$. Now, if $x \in H \cap {\rm PGL}(V)$ then $\nu(x) \geqs 3$ and the required lower bound on $|x^G|$ follows from \cite[Corollary 3.38]{FPR2}. Similarly, if $x \not\in H \cap {\rm PGL}(V)$ then \cite[Corollary 3.49]{FPR2} gives
\[
|x^G|>\frac{1}{2}Qq^{\frac{1}{2}(n^2-n-4)}
\]
(minimal if $n$ is even and $x$ is an involutory graph automorphism with $C_G(x)$ of type ${\rm Sp}_n(q)$) 
and the result follows for $n \geqs 10$ since $n^2-n-4 \geqs 12(n-3)$. Finally, if $n=9$ and $x \not\in H \cap {\rm PGL}(V)$ then by inspecting \cite[Table 3.11]{FPR2} we deduce that $|x^G|>\frac{1}{2}Qq^{39}$ and the result follows.

Similarly, if $G_0 = {\rm PSp}_n(q)$ or ${\rm P\O}_n^{\e}(q)$, then it is routine to check that the desired result follows by combining the bounds presented in \cite[Corollary 3.38]{FPR2} and \cite[Corollary 3.49]{FPR2}. 
\end{proof}

\begin{rem}\label{r:pso}
Suppose $G_0 = {\rm P\O}_n^{\e}(q)$, where $n \geqs 12$ is even, and let $H$ be a subgroup of $G$ with $\nu(H) \geqs 3$. We claim that
\[
|x^G|>\frac{1}{2d}q^{3(n-3)}
\]
for all $x \in H$ of prime order, where $d = (2,q-1)$. This lower bound follows immediately from \cite[Corollary 3.49]{FPR2} if $x \not\in {\rm PGL}(V)$, so let us assume $x \in H \cap {\rm PGL}(V)$ has prime order $r$. If $\nu(x) \geqs 4$ then \cite[Corollary 3.38]{FPR2} gives $|x^G|>\frac{1}{4}Qq^{4(n-5)}$ and the result follows unless $(n,q) = (12,2)$. In the latter case, it is easy to check that the condition $\nu(x) \geqs 4$ implies that $|x^G|>2^{27}$, minimal if $x$ is an involution of type $a_4$ in the notation of \cite{AS}. And if $\nu(x) = 3$ then $r=2$ and either $q$ is even and $|x^G|>\frac{1}{2}q^{3(n-3)}$ (see \cite[Proposition 3.22]{FPR2}), or $q$ is odd, $x$ is an involution of type $(-I_3,I_{n-3})$ and  
\[
|x^G| \geqs \frac{|{\rm SO}_{n}^{\e}(q)|}{2|{\rm SO}_3(q)||{\rm SO}_{n-3}(q)|}>\frac{1}{4}q^{3(n-3)}.
\]
This justifies the claim. 
\end{rem}

\subsection{Probabilistic methods}\label{ss:prob}

Let $G$ be a finite group, let $\tau = (H_1, \ldots, H_t)$ be a $t$-tuple of core-free subgroups and consider the natural action of $G$ on $\Gamma = G/H_1 \times \cdots \times G/H_t$. Let   
\[
Q(G,\tau) = \frac{|\{(\alpha_1, \ldots, \alpha_t) \in \Gamma \,:\, \bigcap_{i = 1}^{t}G_{\alpha_i}\neq 1\}|}{|\Gamma|}
\]
be the probability that a uniformly random point in $\Gamma$ is not contained in a regular orbit of $G$ and observe that $\tau$ is regular if and only if $Q(G, \tau) < 1$. 

Recall that if $x \in G$ and $H$ is a core-free subgroup of $G$, then we write
\[
{\rm fpr}(x,G/H) = \frac{|x^G \cap H|}{|x^G|}
\]
for the \emph{fixed point ratio} of $x$ in its action on $G/H$, which is simply the proportion of points in $G/H$ fixed by $x$. In particular, ${\rm fpr}(x,G/H)$ is the probability that a uniformly random point in $G/H$ is fixed by $x$. 

The following key result is \cite[Lemma 2.1]{AB} and it gives a natural extension of a powerful  probabilistic approach for studying base sizes, which was originally introduced by Liebeck and Shalev in \cite{LSh2}. In particular, it  allows us to bring upper bounds on fixed point ratios into play in order to derive an upper bound on the probability $Q(G,\tau)$.

\begin{lem}\label{l:fpr}
Let $G$ be a finite group and let $\tau = (H_1, \ldots, H_t)$ be a $t$-tuple of core-free subgroups of $G$. Then $\tau$ is regular if 
\[
\widehat{Q}(G,\tau) := \sum_{i=1}^{\ell}|x_i^G| \cdot \left(\prod_{j=1}^t {\rm fpr}(x_i, G/H_j)\right) < 1,
\]
where $x_1, \ldots, x_{\ell}$ is a set of representatives of the conjugacy classes in $G$ of elements of prime order.
\end{lem}

\begin{proof}
By the proof of \cite[Lemma 2.1]{AB} we have $Q(G,\tau) \leqs \widehat{Q}(G,\tau)$ and the result follows immediately.
\end{proof}

\begin{rem}\label{r:conj}
In the special case where each component subgroup $H_j$ in $\tau$ is conjugate to $H_1$, we write 
\[
Q(G,H_1,t) = Q(G,\tau),\;\; \widehat{Q}(G,H_1,t) = \widehat{Q}(G,\tau),
\]
so $Q(G,H_1,t)<1$ if and only if $b(G,H_1) \leqs t$, where we recall that $b(G,H_1)$ denotes the base size with respect to the action of $G$ on $G/H_1$. And in the special case $t=2$, we set 
\[
Q(G,H_1) = Q(G,H_1,2), \;\; \widehat{Q}(G,H_1) = \widehat{Q}(G,H_1,2).
\]
\end{rem}

The next lemma, which is a new observation, will play a key role in the proof of Theorem \ref{t:main}. In the special case $t=2$, this explains why we are interested in determining the pairs $G$ and $H$ in Theorem \ref{t:hat} with $\what{Q}(G,H) \geqs 1$.

\begin{lem}\label{l:key}
Let $G$ be a finite group and let $\tau = (H_1, \ldots, H_t)$ be a $t$-tuple of core-free subgroups of $G$. Then
\[
\widehat{Q}(G,\tau) \leqs \frac{1}{t}\sum_{j=1}^t\widehat{Q}(G,H_j,t).
\]
In particular, $\tau$ is regular if $\widehat{Q}(G,H_j,t)<1$ for all $j$.
\end{lem}

\begin{proof}
Let $x \in G$ be an element of prime order. Then the inequality of arithmetic and geometric means implies that
\[
\prod_{j=1}^t {\rm fpr}(x,G/H_j) \leqs \frac{1}{t}\sum_{j=1}^t {\rm fpr}(x,G/H_j)^t
\]
and thus
\begin{align*}
\widehat{Q}(G,\tau) = \sum_{i=1}^{\ell}|x_i^G| \cdot \left(\prod_{j=1}^t {\rm fpr}(x_i, G/H_j)\right) & \leqs \frac{1}{t} \sum_{i=1}^{\ell} |x_i^G| \cdot \sum_{j=1}^{t} {\rm fpr}(x_i,G/H_j)^t \\
& = \frac{1}{t} \sum_{j=1}^t \left(\sum_{i=1}^{\ell} |x_i^G| \cdot {\rm fpr}(x_i,G/H_j)^t\right) \\
& = \frac{1}{t} \sum_{j=1}^t \what{Q}(G,H_j,t)
\end{align*}
as required.
\end{proof}

\begin{rem}\label{r:key}
Let us observe that we cannot replace $\what{Q}$ by $Q$ in Lemma \ref{l:key}. 
\begin{itemize}\addtolength{\itemsep}{0.2\baselineskip}
\item[{\rm (i)}] Indeed, there exist non-regular tuples $\tau = (H_1, \ldots, H_t)$ such that $b(G,H_j) \leqs t$ for all $j$. For instance, $G = {\rm L}_2(7)$ has two non-conjugate subgroups $H_1$ and $H_2$ such that $H_i \cong A_4$ and $b(G,H_i) = 2$ for $i = 1,2$. Now $H_i$ has a unique regular orbit on $G/H_i$, which means that $Q(G,H_i) = \frac{1}{7}$ for $i = 1,2$. However, we  compute $\what{Q}(G,H_i) = \frac{11}{7}$ for $i = 1,2$ and one can check that $\tau = (H_1,H_2)$ is non-regular. 

\item[{\rm (ii)}] Conversely, there also exist regular tuples $\tau = (H_1, \ldots, H_t)$ such that $b(G,H_j) > t$ for all $j$. For example, $G = A_8$ has subgroups $H_1 = (8{:}7).3$ and $H_2 = 2^2.S_4$ such that $b(G,H_i) = 3$ for $i=1,2$ and $\tau = (H_1,H_2)$ is regular. Here $H_1$ and $H_2$ are both contained in a maximal subgroup ${\rm AGL}_3(2)$ of $G$. 
\end{itemize}
\end{rem}

Our final result is \cite[Lemma 2.2]{AB}, which provides a useful way to produce an upper bound on $\what{Q}(G,\tau)$. We will use this extensively in the proof of Theorem \ref{t:hat}.

\begin{lem}\label{l:favbound}
Let $G$ be a finite group and let $\tau = (H_1, \ldots, H_t)$ be a 
$t$-tuple of core-free subgroups of $G$. Suppose $x_1, \ldots, x_m$ are elements in $G$ and $A_1, \ldots, A_t,B$ are positive real numbers such that
$|x_i^G| \geqs B$ for all $i$ and $\sum_{i = 1}^m |x_i^G\cap H_j| \leqs A_j$ for all $j$. Then 
\[
\sum_{i = 1}^m |x_i^G| \cdot \left( \prod_{j = 1}^t {\rm fpr}(x_i, G/H_j)\right) \leqs B^{1 - t}\cdot \prod_{j = 1}^tA_j.
\]
\end{lem}

\subsection{Computational methods}\label{ss:comp}

In this final preliminary section, we briefly discuss some of the computational methods we use in this paper. All of our computations are performed using {\sc Magma} \cite{magma} (version V2.28-21). Throughout this section, let $G$ be an almost simple classical group with socle $G_0$ and let $H \in \mathcal{S}$ be a maximal subgroup of $G$. 

To begin with, let us assume that we have constructed $G$ and a collection of subgroups $H_i \in \mathcal{S}$ for $i = 1, \ldots, t$, 
working with an appropriate permutation or matrix representation of $G$ in {\sc Magma}. Several methods for making such  constructions will be discussed in Example \ref{ex:con} below.

\begin{ex}\label{ex:qhat}
Setting $H = H_i$, we can use the function \texttt{QHat} presented below in order to calculate $\what{Q}(G,H)$. To do this, we first construct a set of representatives of the conjugacy classes in $H$ comprising elements of prime order and we then determine the fusion of the relevant $H$-classes in $G$. This allows us to compute $|x^G \cap H|$ for every element $x \in G$ of prime order, and we can compute $|x^G|$ by constructing the centraliser $C_G(x)$. Therefore, we can calculate
\[
{\rm fpr}(x,G/H) = \frac{|x^G \cap H|}{|x^G|}
\]
for each $x \in H$ of prime order and then $\what{Q}(G,H)$ is returned. Notice that our function avoids constructing a set of representatives of the conjugacy classes in $G$.

{\small
\begin{verbatim}
QHat:=function(G,H)
  cl:=Classes(H);
  a:=[i : i in [1..#cl] | IsPrime(cl[i][1])];
  b:=[];
  while #a ge 1 do
    c:=[i : i in a | IsConjugate(G,cl[i][3],cl[a[1]][3])];
    Append(~b,&+[cl[i][2]:i in c]^2*#Centraliser(G,cl[a[1]][3])/#G);
    a:=[i : i in a | not i in c];
  end while;
  return &+b;
end function;
\end{verbatim}}
\end{ex}

\begin{ex}\label{ex:random}
Given $G$ and a regular tuple $\tau = (H_1, \ldots, H_t)$, we can often verify that $\tau$ is regular by randomly searching for  elements $x_i \in G$ for $i = 1, \ldots, t-1$ such that
\[
\left(\bigcap_{i=1}^{t-1} H_i^{x_i}\right) \cap H_t = 1.
\]
To do this, we can use the following function, where the input is $G$ and the given tuple $S = [H_1, \ldots, H_t]$. This returns the text ``\texttt{S is regular}" if the random search is successful.
{\small
\begin{verbatim}
RandomReg:=function(G,S)
  repeat
    L:=S[#S];
    for i in [1..#S-1] do
      L:=L meet S[i]^Random(G);
    end for;
  until
    #L eq 1;
   return "S is regular";
end function;    
\end{verbatim}}
\end{ex}

\begin{ex}\label{ex:reg}
In some cases, we will need to show that a given tuple $\tau = (H_1, \ldots, H_t)$ is non-regular. Of course, this conclusion is immediate if 
\[
|G| > \prod_{i=1}^t |G:H_i|.
\]
If this inequality is not satisfied, then it may be feasible to use the function \texttt{RegOrbits} from \cite[Section 1.2]{AB_comp} to calculate the number $r$ of regular orbits of $H_t$ on $G/H_1 \times \cdots \times G/H_{t-1}$, noting that $\tau$ is non-regular if and only if $r=0$. As before, the input is $G$ and the tuple $S = [H_1, \ldots, H_t]$, and the function returns the integer $r$. 
\end{ex}

Finally, let us consider how to construct $G$ and a set of representatives of the relevant conjugacy classes of subgroups in $\mathcal{S}$, which we can use as input in any one of the functions discussed above.

\begin{ex}\label{ex:con}
We consider several concrete examples that will arise later in the paper, in order to illustrate a range of different construction methods.
\begin{itemize}\addtolength{\itemsep}{0.2\baselineskip}
\item[{\rm (i)}] If $G = \O_{14}^{+}(2)$ and $H = A_{16}$, then we can work with the natural matrix representation of $G$ and we can construct $H$ using the \texttt{ClassicalMaximals} function:
{\small
\begin{verbatim}
G:=OmegaPlus(14,2);
H:=ClassicalMaximals("O+",14,2:classes:={9})[1];
\end{verbatim}}

\vspace{1mm}

\noindent We can then use the function \texttt{QHat} defined above to compute 
\[
\hspace{15mm} \what{Q}(G,H) = \frac{26753167521}{4386258944},
\]
which only takes a few seconds. We can also verify that $\tau = (H,H)$ is regular (and thus $b(G,H) = 2$) via \texttt{RandomReg}, which may take several minutes to conclude. Notice that the prohibitively large index $|G:H| = 161695049711616$ means that  the function \texttt{RegOrbits} is not effective in this case (this is not a problem because we do not need to know the precise number of regular orbits of $H$ on $G/H$).

\vspace{1mm}

\item[{\rm (ii)}] Next assume $G = {\rm O}_{10}^{-}(2)$ and $H = {\rm M}_{12}.2$, in which case $H_0 < A_{12} < G_0$ is a non-maximal subgroup of $G_0$. Since $H = N_G(H_0)$ and $H_0$ is a maximal subgroup of $A_{12}$, which in turn is maximal in $G_0$, we can construct $G$ and $H$ as follows:
{\small
\begin{verbatim}
G:=SOMinus(10,2);
C:=ClassicalMaximals("O-",10,2:classes:={9});
M:=MaximalSubgroups(C[1]:OrderEqual:=95040);
H:=Normaliser(G,M[1]`subgroup);
\end{verbatim}}

\vspace{1mm}

\item[{\rm (iii)}] Suppose $G = {\rm PGSp}_{10}(3)$ and $H = {\rm U}_5(2).2$. Here we can use a variation of the previous method to construct $G$ and $H$ as permutation groups of degree $29524$:
{\small
\begin{verbatim}
f,G,K:=PermutationRepresentation(CSp(10,3):ModScalars:=true);
H:=f(ClassicalMaximals("S",10,3:classes:={9},normaliser:=true)[2]);
\end{verbatim}}

\vspace{1mm}

\item[{\rm (iv)}] Now assume $G = {\rm L}_9(2).2$ and $H = {\rm L}_3(4).D_{12}$. Here we can construct $G$ and $H$ as permutation groups of degree $1022$ as follows, using the fact that $G$ has a unique conjugacy class of maximal subgroups of order $|H|$ (see \cite[Table 8.55]{BHR}):
{\small
\begin{verbatim} 
G:=AutomorphismGroupSimpleGroup("L",9,2);
H:=MaximalSubgroups(G:OrderEqual:=12*#PSL(3,4))[1]`subgroup;
\end{verbatim}}

\vspace{1mm}

\item[{\rm (v)}] For our final example, suppose $G_0 = {\rm P\O}_8^{+}(3)$ and $G = G_0.\la \gamma\ra = G_0.2$, where $\gamma$ is an involutory graph automorphism. Now $\mathcal{S}$ contains two conjugacy classes of maximal subgroups isomorphic to ${\rm O}_8^{+}(2) = \O_{8}^{+}(2).2$, with representatives $H$ and $K$, and we can construct $G,H$ and $K$ as permutation groups of degree $3360$ as follows:

\vspace{1mm}

{\small
\begin{verbatim}
G:=LowIndexSubgroups(AutomorphismGroupSimpleGroup("O+",8,3),24)[2];
M:=MaximalSubgroups(G);
C:=[M[i]`subgroup : i in [8,9]];
\end{verbatim}}

\vspace{1mm}

\noindent We can then use the \texttt{RegOrbits} function to verify that every triple (and therefore every pair) of subgroups with components in $\{H,K\}$ is non-regular. And we can use \texttt{RandomReg} to check that every $4$-tuple is regular.
\end{itemize}
\end{ex}

\section{Proof of Theorem \ref{t:hat}}\label{s:hat}

Our goal in this section is to prove Theorem \ref{t:hat}, which provides a major reduction for the proof of Theorem \ref{t:main}. As usual, let $G$ be a finite almost simple classical group over $\mathbb{F}_q$ with socle $G_0$ and natural module $V$. Write $q = p^f$ with $p$ a prime and set $n = \dim V$ and $Q = q/(q+1)$. Recall that we assume $G_0$ is one of the groups in \eqref{e:groups}. In addition, let $H$ be a core-free maximal subgroup of $G$ contained in the collection $\mathcal{S}$ (see Section \ref{ss:class}) and write $H_0$ for the socle of $H$. Define $\nu(H)$ as in \eqref{e:nuH} and set
\[
\what{Q}(G,H) = \sum_{i=1}^{\ell} |x_i^G| \cdot {\rm fpr}(x_i,G/H)^2 = \sum_{i=1}^{\ell} \frac{|x_i^G \cap H|^2}{|x_i^G|}, 
\]
where $x_1, \ldots, x_{\ell}$ is a complete set of representatives of the conjugacy classes in $G$ of elements of prime order. It will also be helpful to recall Lemma \ref{l:favbound}, which we will use repeatedly.

We will need the following technical lemma. In parts (ii) and (iii), we write $J_k$ for a standard unipotent Jordan block of size $k$, and the notation $x = (J_n^{a_n}, \ldots, J_1^{a_1})$ indicates that the Jordan form of $x$ on $V$ comprises $a_i$ unipotent Jordan blocks of size $i$, for $i = 1, \ldots, n$.

\begin{lem}\label{l:sp6case}
Let $G_0 = {\rm P\O}_n^{\e}(q)$ with $n = 14-\delta_{3,p}$ and let $H \in \mathcal{S}$ with socle $H_0 = {\rm PSp}_6(q)$. Let $x \in H$ be an element of prime order $r$ and let $W$ be the natural module for $H_0$.
\begin{itemize}\addtolength{\itemsep}{0.2\baselineskip}
\item[{\rm (i)}] We have $\nu(H) \geqs 4$.
\item[{\rm (ii)}] If $n = 14$ then either $|x^G|>q^{44}$, or $x \in H \cap {\rm PGL}(V)$, $r=p$,  $x = (J_2,J_1^4)$ on $W$ and $x = (J_2^4,J_1^6)$ on $V$.

\item[{\rm (iii)}] If $n = 13$ then either $|x^G|>q^{44}$, or one of the following holds:

\vspace{1mm}

\begin{itemize}\addtolength{\itemsep}{0.2\baselineskip}
\item[(a)] $r=2$ and $|x^G|>\frac{1}{4}q^{39}$.
\item[(b)] $r=p$ and either $x = (J_2,J_1^4)$ on $W$ and $x = (J_2^4,J_1^5)$ on $V$, or $\nu(x) \geqs 6$ and $|x^G|>\frac{1}{4}Qq^{42}$.
\end{itemize}
\end{itemize} 
\end{lem}

\begin{proof}
First observe that $V$ is the unique nontrivial composition factor of the $15$-dimensional module $\L^2(W)$. By Theorem \ref{gursax} we have $\nu(H) \geqs 3$. Recall that ${\rm Inndiag}(G_0)$ is the subgroup of ${\rm Aut}(G_0)$ generated by the inner and diagonal automorphisms of $G_0$.

To begin with, let us assume $n = 14$. Here the conditions recorded in \cite[Tables 11.0.14, 11.0.16]{Schroder} imply that $G \cap {\rm PGL}(V) \leqs {\rm Inndiag}(G_0)$ and thus $\nu(H) \geqs 4$ since there are no prime order elements $x \in {\rm Inndiag}(G_0)$ with $\nu(x) = 3$. In addition, if $x \not\in H \cap {\rm PGL}(V)$, then $x$ is a field or graph-field automorphism, so $q \geqs 4$ and \cite[Corollary 3.49]{FPR2} gives
\[
|x^G| > \frac{1}{8}q^{\frac{1}{4}n(n-1)} = \frac{1}{8}q^{91/2} \geqs q^{44}
\]
as required. 

Next assume $x \in H \cap {\rm PGL}(V)$ has order $r=p$. By calculating with the exterior-square $\Lambda^2(W)$, it is straightforward to show that the transvections in $H$ have Jordan form $(J_2^4,J_1^6)$ on $V$ (moreover, if $p=2$ then $x$ is an $a_4$-type involution, in the terminology of \cite{AS}; also see \cite[Section 3.5.4]{BG}). And in all other cases, either $p=2$ and $x$ has Jordan form $(J_2^6,J_1^2)$, or $p \geqs 3$ and $|x^G|$ is minimal when $x$ has Jordan form $(J_2^2,J_1^2)$ on $W$, in which case $x$ acts on $V$ as $(J_3,J_2^4,J_1^3)$. In the latter case, we can use \cite[Lemma 3.18]{FPR2} to show that $|x^G|>q^{44}$. And for $p=2$ we find that $x$ is a $c_6$-type involution, so $|x^G|>\frac{1}{2}q^{48}$ by \cite[Proposition 3.22]{FPR2}.

Similarly, if $x \in H \cap {\rm PGL}(V)$ is semisimple of prime order, then it is straightforward to show that either $\nu(x) = 6$ and thus $|x^G|>\frac{1}{4}q^{48}$ (minimal if $p\ne 2$ and $x \in H$ is an involution of type $(-I_2,I_4)$, which acts on $V$ as $(-I_8,I_6)$), or $\nu(x) \geqs 8$ and we have $|x^G|>q^{44}$.

To complete the proof, let us assume $n = 13$, so $p=3$. If $x$ is a field automorphism of prime order $r$, then \cite[Lemma 3.48]{FPR2} gives $|x^G|>\frac{1}{4}q^{52}$ if $r \geqs 3$ and we get $|x^G|>\frac{1}{4}q^{39}$ if $r=2$. Now suppose $x \in H$ is unipotent of order $3$. Here we calculate that $x = (J_2,J_1^4)$ has Jordan form $(J_2^4,J_1^5)$ on $V$. And for all other unipotent elements, we find that $|x^G|$ is minimal when $x$ has Jordan form $(J_2^2,J_1^2)$, in which case $x$ acts on $V$ as $(J_3,J_2^4,J_1^2)$ and we get $|x^G|>\frac{1}{4}Qq^{42}$. Finally, suppose $x \in H$ is semisimple of prime order $r$. If $r=2$ then $\nu(x) \geqs 5$ and thus $|x^G|>\frac{1}{4}q^{40}$. And if $r \geqs 5$ is odd, then $\nu(x) \geqs 6$ and $|x^G|>\frac{1}{2}Qq^{48}$.
\end{proof}

\begin{lem}\label{l:t2_14}
the conclusion of Theorem \ref{t:hat} holds if $n \geqs 14$.
\end{lem}

\begin{proof}
First assume $n \geqs 15$. By the main theorem of \cite{BGS} we have $b(G,H) = 2$ and by carefully inspecting the proof, one can check that the argument gives $\what{Q}(G,H) < 1$ in all cases. 

For example, suppose $G = {\rm Sp}_n(2)$ and $H = S_{n+2}$, where $n \geqs 16$ and $n \equiv 0 \imod{4}$. Here we can identify the natural module $V$ for $G$ with the fully deleted permutation module for $H$ over $\mathbb{F}_2$ (see the first row of Table \ref{atab}). For $n \geqs 28$, the proof of \cite[Lemma 3.2]{BGS} shows that
\[
{\rm fpr}(x,G/H) < |x^{G}|^{-\frac{2}{3}}
\]
for all $x \in H$ of prime order. Therefore, 
\[
\what{Q}(G,H) = \sum_{i=1}^{\ell} |x_i^G| \cdot {\rm fpr}(x_i,G/H)^2 < \sum_{i=1}^{\ell} |x_i^G|^{-\frac{1}{3}}
\]
and thus \cite[Proposition 2.2]{B07} implies that $\what{Q}(G,H) <1$ as required. And the claim for $n \in \{16,20,24\}$ can be checked directly, as explained in the proof of \cite[Lemma 3.2]{BGS}.

So to complete the proof of the lemma, we may assume $n=14$, noting that the subgroups in $\mathcal{S}$ are determined (up to conjugacy) in \cite{Schroder}, with convenient tables presented in \cite[Chapter 11]{Schroder}.

Suppose $G_0 = {\rm L}_{14}^{\e}(q)$ is a linear or unitary group. We have $\nu(H) \geqs 3$ by Theorem \ref{gursax}, so $|x^G|>\frac{1}{2}Qq^{66}$ for all $x \in H$ of prime order by Proposition \ref{p:xgbd}. In addition, by inspecting \cite[Tables 11.0.8, 11.0.10]{Schroder} we see that $q \geqs 3$ and $|H| \leqs |{\rm Sp}_6(3)|$. Therefore, Lemma \ref{l:favbound} implies that
\[
\what{Q}(G,H) < 2|{\rm Sp}_6(3)|^2Q^{-1}q^{-66} < 1
\]
for all $q \geqs 3$. The result follows.

The case $G_0 = {\rm PSp}_{14}(q)$ is very similar. By combining Theorem \ref{gursax} with Proposition \ref{p:xgbd}, we see that $|x^G|>\frac{1}{4}Qq^{33}$ for all $x \in H$ of prime order. So if $|H|<q^{15}$, then 
\[
\what{Q}(G,H) < 4Q^{-1}q^{-3} < 1
\]
for all $q \geqs 2$. By inspecting \cite[Table 11.0.12]{Schroder}, we may assume $H_0 = {\rm PSp}_6(q)$ with $q$ odd. This is the case labelled $(\mathcal{D}6)$ in \cite[Table 11]{BGS} and the proof of \cite[Lemma 6.6]{BGS} shows that $\what{Q}(G,H) <1$.

To complete the proof of the lemma, we may assume $G_0 = {\rm P\O}_{14}^{\e}(q)$. If $G_0 = \O_{14}^{+}(2)$ and $H_0 = A_{16}$, then $b(G,H) = 2+\a$, where $\a=1$ if $G = {\rm O}_{14}^{+}(2)$ and $H = S_{16}$, otherwise $\a = 0$. And for $G = G_0$ we can use {\sc Magma} \cite{magma} (see Example \ref{ex:con}(i)) to show that
\[
\what{Q}(G,H) = \frac{26753167521}{4386258944} > 1,
\]
which explains why the pair $(G,H) = (\O_{14}^{+}(2),A_{16})$ is recorded in Table \ref{tab:new}. In the same way, we can check that $\what{Q}(G,H)<1$ for all the remaining cases with $q=2$. 

Now assume $q \geqs 3$. If $H_0 = A_{15}$ then $q \geqs 7$ and the proof of \cite[Lemma 3.5]{BGS} gives $\what{Q}(G,H)<1$. Therefore, we may assume $H \not\in \mathcal{A}$ (see Table \ref{atab}), so $\nu(H) \geqs 3$ by Theorem \ref{gursax} and thus the bound $|x^G|>\frac{1}{2d}q^{33}$ in Remark \ref{r:pso} holds for all $x \in H$ of prime order, where $d = (2,q-1)$. So if 
$|H|<2\log_pq \cdot q^{15}$ then
\[
\what{Q}(G,H) < 8d(\log_pq)^2q^{-3} < 1
\]
for all $q \geqs 3$. 

Finally, suppose $q \geqs 3$ and $|H| \geqs 2\log_pq \cdot q^{15}$. By inspecting \cite[Tables 11.0.14, 11.0.16]{Schroder} we deduce that $H_0 = {\rm PSp}_6(q)$ and the conditions in the relevant tables indicate that $q \geqs 4$. This case is listed as $(\mathcal{D}7)$ in \cite[Table 11]{BGS} and the main theorem of \cite{BGS} gives $b(G,H) = 2$. However, the details for this case were omitted in the proof of \cite[Lemma 6.6]{BGS} and so we need to address this omission here in order to verify the bound $\what{Q}(G,H)<1$. 

Now $|H|<\log_2q\cdot q^{21}$ and so Lemma \ref{l:favbound} implies that the contribution to $\widehat{Q}(G,H)$ from the elements $x \in G$ with $|x^G|>q^{44}$ is less than $\a = (\log_2q)^2q^{-2}$. Now assume $x \in H$ has prime order and $|x^G| \leqs q^{44}$. Then Lemma \ref{l:sp6case}(ii) implies that $x \in H \cap {\rm PGL}(V)$ is unipotent with Jordan form $(J_2,J_1^4)$ on the natural module for $H_0$. Here $|x^G|>\frac{1}{4}Qq^{36}$ and there are fewer than $q^6$ such elements in $H$, so their contribution to $\what{Q}(G,H)$ is less than
\[
q^{12} \cdot 4Q^{-1}q^{-36} = 4Q^{-1}q^{-24} = \b.
\]
Bringing these estimates together, we conclude that 
$\what{Q}(G,H) < \a+\b<1$ for all $q \geqs 4$ and the result follows.
\end{proof}

\begin{lem}\label{l:t2_13}
the conclusion of Theorem \ref{t:hat} holds if $n =13$.
\end{lem}

\begin{proof}
First assume $G_0 = {\rm L}_{13}^{\e}(q)$ and let $H \in \mathcal{S}$. Then $|x^G|>\frac{1}{2}Qq^{60}$ for all $x \in H$ of prime order (via Theorem \ref{gursax} and Proposition \ref{p:xgbd}) and by inspecting \cite[Tables 11.0.2, 11.0.4]{Schroder} we see that $|H| \leqs |{\rm Sp}_6(3)|$. Therefore, Lemma \ref{l:favbound} yields
\[
\what{Q}(G,H) < 2|{\rm Sp}_6(3)|^2Q^{-1}q^{-60} < 1
\]
for all $q \geqs 3$. Finally, if $q = 2$ then $G_0 = {\rm U}_{13}(2)$,  $H_0 = {\rm PSp}_6(3)$ and the proof of \cite[Lemma 4.4]{BGS} shows that $\what{Q}(G,H)<1$.

For the remainder, we may assume $G_0 = \O_{13}(q)$. For $q=3$ we can use {\sc Magma} (as in Example \ref{ex:con}(i)) to verify the bound $\widehat{Q}(G,H)<1$ for all $H \in \mathcal{S}$. Now assume $q \geqs 5$. If $H \in \mathcal{A}$, then $H_0 = A_{15}$ or $A_{14}$ and the proof of \cite[Lemma 3.5]{BGS} gives $\what{Q}(G,H)<1$. In the remaining cases, Theorem \ref{gursax} gives $\nu(H) \geqs 3$ and thus $|x^G|>\frac{1}{4}Qq^{27}$ for all $x \in H$ of prime order by Proposition \ref{p:xgbd}. If $|H| < \log_3q\cdot q^{10}$, we deduce that 
\[
\what{Q}(G,H) < 4Q^{-1}(\log_3q)^2\cdot q^{-7} < 1
\]
for all $q \geqs 5$. So by inspecting \cite[Table 11.0.6]{Schroder}, we see that we may assume $H_0 = {\rm PSp}_6(q)$ and $p=3$. This corresponds to the case labelled $(\mathcal{D}7)$ in \cite[Table 11]{BGS}, but the details for this case were omitted in the proof of \cite[Lemma 6.6]{BGS}.

As in the proof of Lemma \ref{l:t2_14}, the contribution to $\what{Q}(G,H)$ from the prime order elements $x \in G$ with $|x^G|>q^{44}$ is less than $\a = (\log_3q)^2q^{-2}$. In view of Lemma \ref{l:sp6case}(iii), using the fact that $i_2(H) +i_p(H) < 2q^{18}$, 
we see that the remaining elements with $|x^G|>\frac{1}{4}q^{39}$ contribute less than $\b = (2q^{18})^2 \cdot 4q^{-39} = 16q^{-3}$ (recall that we write $i_r(H)$ for the number of elements in $H$ of order $r$). Finally, if $x$ has Jordan form $(J_2^4,J_1^5)$ on $V$ then $|x^G|>\frac{1}{4}q^{32}$ and Lemma \ref{l:sp6case}(iii) implies that there are no more than $q^6$ such elements in $H$, so the final contribution is no more than $\gamma = 4q^{-20}$. It is now a routine exercise to check that $\what{Q}(G,H)<\a+\b+\gamma<1$ for all $q \geqs 5$.
\end{proof}

\begin{lem}\label{l:t2_12}
the conclusion of Theorem \ref{t:hat} holds if $n =12$.
\end{lem}

\begin{proof}
First assume $G_0 = {\rm L}_{12}^{\e}(q)$. Here $\nu(H) \geqs 3$ and $|H| \leqs 2|{\rm Suz}|$ (see \cite[Tables 8.77, 8.79]{BHR}), so 
$|x^G|>\frac{1}{2}Qq^{54}$ for all $x \in H$ of prime order and it follows that 
\[
\what{Q}(G,H) < 8|{\rm Suz}|^2Q^{-1}q^{-54} < 1
\]
for all $q \geqs 3$. Now assume $q = 2$, so $G_0 = {\rm U}_{12}(2)$ and $H_0 = {\rm Suz}$ is the only possibility (more precisely, we have $G = G_0.c$ and $H = H_0.c$ with $c \in \{1,2\}$). This is the case labelled $(\mathcal{B}9)$ in \cite[Table 4]{BGS} and we have $b(G,H) = 2$ by \cite[Lemma 4.4]{BGS}. However, the proof of the latter lemma omits the details for this case, so some additional work is required in order to verify the bound $\what{Q}(G,H)<1$.

If $x \in G$ is an involutory graph automorphism, then $|x^G|>2^{63}=b_1$ and we note that $i_2(H \setminus H_0) < 2^{22}=a_1$. Now assume $x \in H_0$ has prime order $r$. Using the {\sc Magma} function \texttt{ClassicalMaximals}, we can construct the covering group $\widehat{H}_0 = 3.{\rm Suz}$ as a subgroup of ${\rm SU}_{12}(2)$ and this allows us to calculate the action of $x$ on $V$. If $r = 2$ then $|x^G| > 2^{63} = b_2$ (minimal if $x$ is in the \texttt{2A} class of $H_0$, which has Jordan form $(J_2^4,J_1^4)$ on $V$) and we note that $i_2(H_0) = 2915055 < 2^{22}=a_2$. Similarly, if $r=3$ then $\nu(x) \geqs 6$, so $|x^G|> 2^{71}=b_3$ and we note that $i_3(H) < 2^{28}=a_3$. And if $r \geqs 5$ then $\nu(x) \geqs 8$, so
\[
|x^G| \geqs \frac{|{\rm GU}_{12}(2)|}{|{\rm GU}_{4}(2)|^3} > 2^{95} = b_4
\]
and we observe that $|H_0|<2^{39} = a_4$. Bringing these estimates together, we conclude that
\begin{equation}\label{e:bdnew}
\what{Q}(G,H) < \sum_{i=1}^4 a_i^2b_{i}^{-1} < 1
\end{equation}
as required. 

Next assume $G_0 = {\rm PSp}_{12}(q)$. If $q = 2$, then $G = {\rm Sp}_{12}(2)$ and $H = S_{14}$ or ${\rm L}_2(25).2$. In both cases, the main theorem of \cite{BGS} gives $b(G,H) = 2$ and we can use {\sc Magma} (as in Example \ref{ex:con}(i)) to compute $\widehat{Q}(G,H)$. For $H = {\rm L}_2(25).2$ we get $\widehat{Q}(G,H)< 1$, whereas 
\[
\what{Q}(G,H) = \frac{12927675791}{1554186240}
\]
for $H = S_{14}$, so the latter case is recorded in Table \ref{tab:new}. 

Next assume $q = 3$. Here $H_0 = {\rm Suz}$ is the only possibility, which is case $(\mathcal{E}3)$ in \cite[Table 12]{BGS}. By \cite[Proposition 6.8]{BGS} we see that $b(G,H) = 2$, but once again the details are omitted in the proof and so some additional work is required.

If $x \in G \setminus G_0$ is an involution, then 
\[
|x^G| \geqs \frac{|{\rm Sp}_{12}(3)|}{2|{\rm Sp}_6(3^2)|} > 3^{35} = b_1 
\]
and we note that $i_2(H \setminus H_0) < 3^{14}=a_1$. Now assume $x \in H_0$ has prime order $r$. As above, we can use {\sc Magma} to construct $\widehat{H}_0 = 2.{\rm Suz}$ as a subgroup of ${\rm Sp}_{12}(3)$ and this allows us to determine the action of $x$ on $V$. If $r=2$ then $|x^G| > 3^{32} = b_2$ (minimal if $x$ is in the \texttt{2A} class, which acts on $V$ as $(-I_4,I_8)$) and we note that $i_2(H_0) = 2915055 < 3^{14}=a_2$. Similarly, if $r=3$ then $\nu(x) \geqs 6$, so $|x^G|> 3^{41}=b_3$ and we have $i_3(H) < 3^{18}=a_3$. And if $r \geqs 5$ then $\nu(x) \geqs 8$, so
\[
|x^G| \geqs \frac{|{\rm Sp}_{12}(3)|}{|{\rm Sp}_{4}(3)||{\rm GU}_4(3)|} > 3^{51} = b_4
\]
and we note that $|H_0|<3^{25} = a_4$. It is now routine to check that the upper bound in \eqref{e:bdnew} is satisfied.

Finally, if $q \geqs 4$ then by inspecting \cite[Table 8.81]{BHR}, one can check that $|H|<q^{13 - \delta_{4,q}}$ and the result follows since $\nu(H) \geqs 3$ and thus $|x^G|>\frac{1}{4}Qq^{27}$ for all $x \in H$ of prime order by Proposition \ref{p:xgbd}.

To complete the proof of the lemma, let $G_0 = {\rm P\O}_{12}^{\e}(q)$. First assume $q = 2$, so $\e = -$ and $H_0 = A_{13}$ or ${\rm L}_3(3)$. If $H_0 = A_{13}$ then $b(G,H) = 3$, as recorded in \cite[Table 1]{BGS}, and for $H_0 = {\rm L}_3(3)$ we can use {\sc Magma} (as in Example \ref{ex:con}(i)) to check that $\what{Q}(G,H)<1$. Now assume $q \geqs 3$. If $H \in \mathcal{A}$, then $H_0 = A_{14}$ or $A_{13}$, $q$ is odd and the proof of \cite[Lemma 3.5]{BGS} gives $\what{Q}(G,H)<1$. Now assume $H \not\in \mathcal{A}$. Here $q \geqs 4$ and Theorem \ref{gursax} gives $\nu(H) \geqs 3$, so $|x^G|>\frac{1}{4}q^{27}$ for all $x \in H$ of prime order (see Remark \ref{r:pso}). In addition, we have $|H|<q^8$ and the result quickly follows.
\end{proof}

\begin{lem}\label{l:t2_11}
the conclusion of Theorem \ref{t:hat} holds if $n =11$.
\end{lem}

\begin{proof}
First assume $G_0 = {\rm L}_{11}^{\e}(q)$. For $q = 2$ we have $\e=+$ and $H_0 = {\rm L}_{2}(23)$ or ${\rm M}_{24}$. If $H_0 = {\rm L}_{2}(23)$ then $|H|<q^{14}$ and by combining Theorem \ref{gursax} with Proposition \ref{p:xgbd} we deduce that $|x^G|>\frac{1}{2}Qq^{48}$ for all $x \in H$ of prime order, which yields $\what{Q}(G,H)<1$. And for $H_0 = {\rm M}_{24}$ we can use {\sc Magma} to verify the bound $\what{Q}(G,H)<1$, following Example \ref{ex:con}(i). Now assume $q \geqs 3$. By arguing as above, we have $|x^G|>\frac{1}{2}Qq^{48}$ for all $x \in H$ of prime order and we deduce that $\what{Q}(G,H)<1$ since $|H| \leqs 2|{\rm U}_5(2)|$.

Now assume $G_0 = \O_{11}(q)$. If $H \in \mathcal{A}$ then $H_0 = A_{13}$ or $A_{12}$ and the proof of \cite[Lemma 3.5]{BGS} gives $\what{Q}(G,H)<1$. In the remaining cases, we observe that $q \geqs 13$, $|H|<q^4$ and $\nu(H) \geqs 3$, so $|x^G|>\frac{1}{4}Qq^{21}$ by Proposition \ref{p:xgbd} and the result follows.
\end{proof}

\begin{lem}\label{l:t2_10}
the conclusion of Theorem \ref{t:hat} holds if $n =10$.
\end{lem}

\begin{proof}
We begin by assuming $G_0 = {\rm L}_{10}^{\e}(q)$. We have $\nu(H) \geqs 3$ and thus $|x^G|>\frac{1}{2}Qq^{42}$ for all $x \in H$ of prime order. In particular, if $|H|<q^{20}$ then $\what{Q}(G,H)<2Q^{-1}q^{-2}<1$, so we may assume $|H| \geqs q^{20}$. By inspecting \cite[Tables 8.61, 8.63]{BHR}, we see that $H_0 = {\rm L}_5^{\e}(q)$ is the only possibility. This is included in the case labelled $(\mathcal{B}1)$ in \cite[Table 4]{BGS}  and the proof of \cite[Lemma 4.2]{BGS} shows that $\what{Q}(G,H)<1$.

Next suppose $G_0 = {\rm PSp}_{10}(q)$. To begin with, let us assume $H_0 \ne {\rm U}_5(2)$. Then $\nu(H) \geqs 3$ by Theorem \ref{gursax} and thus $|x^G|>\frac{1}{4}Qq^{21}$ for all $x \in H$ of prime order (see Proposition \ref{p:xgbd}). The result now follows since $|H|<q^7$. Now assume $H_0 = {\rm U}_5(2)$, so $q$ is an odd prime. This is the case labelled $(\mathcal{C}14)$ in \cite[Table 5]{BGS} and so it is an exception to the main statement of Theorem \ref{gursax}. By consulting \cite[Lemma 5.4]{BGS}, we see that $b(G,H) =2$, but the proof of this lemma does not give any details in this specific case, so some additional work is needed in order to verify the bound $\what{Q}(G,H)<1$. Let $\chi$ be the Brauer character corresponding to the associated representation $\rho:H_0 \to {\rm GL}(V)$.

First we claim that $\nu(H) = 2$. To see this, first note that if $x \in H_0$ is an involution with Jordan form $(J_2,J_1^3)$ on the natural module for $H_0$, then $\chi(x) = -6$. This means that $x$ acts as $(-I_8,I_2)$ on $V$ and thus $\nu(x) = 2$. So in order to conclude that $\nu(H) = 2$, it just suffices to show that no element $x \in H \cap {\rm PGL}(V)$ of order $p$ has  Jordan form $(J_2,J_1^8)$ on $V$, so we may assume $q = p \in \{3,5,11\}$. With the aid of {\sc Magma}, we calculate that $x$ has Jordan form $(J_{10})$ when $q = 11$, and $(J_5^2)$ when $q = 5$. Similarly, if $q=3$ then $H_0$ has $6$ conjugacy classes of elements of order $3$ and we see that $\nu(x) \geqs 4$ for all $x \in H$ of order $3$. This justifies the claim. 

If $q \geqs 11$ then $|H|<q^{15/2}$ and by arguing as in the proof of Proposition \ref{p:xgbd} we observe that $|x^G|>\frac{1}{4}Qq^{16}$ for all $x \in H$ of prime order, using the fact that $\nu(H) = 2$. This implies that $\what{Q}(G,H) < 4Q^{-1}q^{-1}<1$ and the result follows.

Next assume $q \in \{5,7\}$. Here $|H|<q^{11}=a_1$ and we have $|x^G|>\frac{1}{2}q^{24} = b_1$ for all elements $x \in G$ of prime order with $\nu(x) \geqs 4$. If $x \in H$ has prime order and $\nu(x) < 4$, then $x$ is an involution with Jordan form $(J_2,J_1^3)$ on the natural module for $H_0$, which acts on $V$ as $(-I_8,I_2)$. Here $|x^G \cap H| = 165 = a_2$ and $|x^G|>\frac{1}{2}q^{16} = b_2$, which means that $\what{Q}(G,H) < a_1^2b_1^{-1} + a_2^2b_2^{-1} < 1$. 

Finally, let us assume $q = 3$. Here it is convenient to use {\sc Magma} to check that $\what{Q}(G,H)<1$. For example, if $G = {\rm PGSp}_{10}(3) = G_0.2$ then $H = {\rm U}_5(2).2$ and we can construct $G$ and $H$ as explained in Example \ref{ex:con}(iii).

To complete the proof of the lemma, let us assume $G_0 = {\rm P\O}_{10}^{\e}(q)$. To begin with, suppose $q = 2$, so $\e=-$ and $H_0 = A_{12}$ or ${\rm M}_{12}$. If $H_0 = A_{12}$ then \cite[Theorem 1]{BGS} gives $b(G,H) = 3+\a$, where $\a = 1$ if $G = {\rm O}_{10}^{-}(2)$, otherwise $\a=0$. Now assume $H_0 = {\rm M}_{12}$, in which case $G = {\rm O}_{10}^{-}(2)$, $H = {\rm M}_{12}.2$ and $H_0 < A_{12} < G_0$ is a second maximal subgroup of $G_0$. This is the case labelled $(\mathcal{E}12)$ in \cite[Table 12]{BGS}, but no details are given in the proof of \cite[Proposition 6.8]{BGS}. By constructing $G$ and $H$ as matrix groups in the natural representation (see  Example \ref{ex:con}(ii) for the details), we can use the {\sc Magma} function \texttt{QHat} (see Example \ref{ex:qhat}) to show that  
\[
\what{Q}(G,H) = \frac{49233}{152320}.
\]

Next assume $q = 3$. If $\e = +$ then $H_0 = A_{12}$ and the proof of \cite[Lemma 3.5]{BGS} shows that $b(G,H) = 2$ via random search. In fact, we can use {\sc Magma} to show that $\widehat{Q}(G,H)<1$. More precisely, for $G_0 = \O_{10}^{+}(3)$ we use \texttt{QHat} (see Example \ref{ex:qhat}) to compute
\[
\what{Q}(G,H) = \frac{14552}{85293},
\]
whereas for $G = G_0.2$ (extended by an involutory graph automorphism) we get 
\[
\what{Q}(G,H) = \frac{2293870}{3497013}.
\]
In the latter case, we can identify $G = {\rm PO}_{10}^{+}(3)$ with an index-two subgroup of ${\rm O}_{10}^{+}(3) = 2 \times {\rm PO}_{10}^{+}(3)$ and this allows us to construct $G$ and $H$ as follows:
{\small
\begin{verbatim} 
G:=NormalSubgroups(GOPlus(10,3):IndexEqual:=2)[1]`subgroup;
H:=Normaliser(G,ClassicalMaximals("O+",10,3:classes:={9})[1]);
\end{verbatim}}

Now, if $\e = -$ then $H_0 = {\rm PSp}_4(q)$ and this is case $(\mathcal{D}5)$ in \cite[Table 11]{BGS}. Here we refer the reader to the proof of \cite[Lemma 6.6]{BGS}. For the remainder, we may assume $q \geqs 5$ is odd. If $H_0 = A_{11}$ then $H \in \mathcal{A}$ and the result follows from the proof of \cite[Lemma 3.5]{BGS}. Similarly, if $H_0 = {\rm PSp}_4(q)$ then we can appeal to the proof of \cite[Lemma 6.6]{BGS}. In each of the remaining cases, it is easy to check that $|H|<q^8$ and we have $|x^G|>\frac{1}{4}Qq^{18}$ since $\nu(H) \geqs 3$ by Theorem \ref{gursax}. Therefore, $\what{Q}(G,H) < 4Q^{-1}q^{-2} < 1$ as required.
\end{proof}

\begin{lem}\label{l:t2_9}
the conclusion of Theorem \ref{t:hat} holds if $n =9$.
\end{lem}

\begin{proof}
To begin with, let us assume $G_0 = {\rm L}_{9}^{\e}(q)$ with $q \geqs 3$. Here $\nu(H) \geqs 3$ by Theorem \ref{gursax} and thus $|x^G|>\frac{1}{2}Qq^{36}$ for all $x \in H$ of prime order. The result now follows since $|H|<q^{35/2}$ (see \cite[Tables 8.55, 8.57]{BHR}). 

Next assume $G_0 = {\rm L}_9(2)$. Here $H_0 = {\rm L}_3(4)$ and we can use {\sc Magma} to check that $\what{Q}(G,H)<1$. For example, if $G = G_0.2$ then we can construct $G$ and $H$ as in Example \ref{ex:con}(iv) and we can then use \texttt{QHat} to verify the desired bound.

Now assume $G_0 = {\rm U}_9(2)$, in which case $H_0 = {\rm J}_3$, ${\rm L}_3(4)$ or ${\rm L}_{2}(19)$. Here we need to adjust our approach because the {\sc Magma} function 
\texttt{MaximalSubgroups} is not effective for groups with socle ${\rm U}_9(2)$. 

Suppose $H_0 = {\rm J}_3$, which is the case labelled $(\mathcal{B}8)$ in \cite[Table 4]{BGS}. By \cite[Lemma 4.4]{BGS} we have $b(G,H) = 2$, but the details for this case were omitted in the proof. In order to show that $\what{Q}(G,H)<1$, let us first observe that $G = G_0.c$ and $H=H_0.c$ with $c \in \{1,2\}$. We can use the function \texttt{ClassicalMaximals} to construct the covering group $\widehat{H}_0 = 3.{\rm J}_3$ as a subgroup of ${\rm SU}_9(2)$, which allows us to determine how each element $x \in H_0$ of prime order acts on $V$. In this way, we see that if $x \in H_0$ is an involution, then $x$ acts on $V$ as $(J_2^4,J_1)$, so $|x^G|>2^{39}=b_1$ and we note that $i_2(H_0) = 26163 = a_1$. Now assume $x \in H_0$ has odd prime order $r$. If $r = 3$ then 
\[
|x^G| \geqs \frac{|{\rm GU}_9(2)|}{|{\rm GU}_1(2)||{\rm GU}_4(2)|^2} > 2^{47} = b_2
\]
and we compute $i_3(H_0) = 253232=a_2$. Similarly, if $r \geqs 5$ then $|x^G|>2^{63} = b_3$ and we note that $|H_0| < 2^{26} = a_3$. Finally, if $x \in G$ is an involutory graph automorphism, then $|x^G| = |G_0|/|{\rm Sp}_8(2)| > 2^{42} = b_4$ and we compute $i_2(H \setminus H_0) = 20520 = a_4$. It is now easy to check that the bound in \eqref{e:bdnew} is satisfied.

Next assume $G_0 = {\rm U}_9(2)$ and $H_0 = {\rm L}_3(4)$. 
Here $(G,H) = (G_0,H_0.6)$ or $(G_0.2,H_0.D_{12})$. Now $6|H_0| < 2^{17} = a_1$ and by considering the embedding of ${\rm SL}_3(4).6$ in ${\rm SU}_9(2)$ we find that $|x^G| > 2^{35} = b_1$ for all $x \in H \cap G_0$ of prime order (minimal if $x$ is an involutory graph-field automorphism of $H_0$, which acts on $V$ with Jordan form $(J_2^3,J_1^3)$). Finally, if $x \in G$ is an involutory graph automorphism then $|x^G|>2^{42} = b_2$ as above and we note that $i_2(H \setminus H_0.6) = 1368=a_2$. It follows that $\what{Q}(G,H) < a_1^2b_1^{-1} + a_2^2b_2^{-1}<1$.

Finally, if $G_0 = {\rm U}_9(2)$ and $H_0 ={\rm L}_2(19)$, then $|H|<2^{13}$ and $\nu(H) \geqs 3$, whence $|x^G|>2^{35}$ for all $x \in H$ of prime order and we conclude that $\what{Q}(G,H)< 2^{-9}$. 

To complete the proof of the lemma, let us assume $G_0 = \O_{9}(q)$. First suppose $q = 3$, so $H_0 = A_{10}$ is the only possibility. In the proof of \cite[Lemma 3.5]{BGS}, random search is used to show that $b(G,H) = 2$, without calculating an upper bound on $\what{Q}(G,H)$. To remedy this, we can proceed as in Example \ref{ex:con}(i). For instance, if $G = G_0.2$ then $H = S_{10}$ and we can construct $G$ and $H$ as follows:
{\small
\begin{verbatim}
G:=SO(9,3);
H:=ClassicalMaximals("O",9,3:classes:={9},special:=true)[1];
\end{verbatim}}
\noindent We can then use the function \texttt{QHat} (see Example \ref{ex:qhat}) to calculate that
\[
\what{Q}(G,H) = \frac{19602287}{20982078}.
\]
Similarly, it is easy to check that $\what{Q}(G,H)<1$ when $G = G_0$.

Now assume $q \geqs 5$. If $H \in \mathcal{A}$ then we can appeal to the proof of \cite[Lemma 3.5]{BGS}. In the remaining cases, we have $|H|<q^7$ and $\nu(H) \geqs 3$, so Proposition \ref{p:xgbd} gives $|x^G|>\frac{1}{4}Qq^{15}$ for all $x \in H$ of prime order and the result follows.
\end{proof}

\begin{lem}\label{l:t2_8}
the conclusion of Theorem \ref{t:hat} holds if $n =8$ and $G_0 \ne {\rm P\O}_8^{+}(q)$.
\end{lem}

\begin{proof}
First assume $G_0 = {\rm L}_{8}^{\e}(q)$, in which case $H_0 = {\rm L}_3(4)$ and $q \geqs 5$. If $x \in H \cap {\rm PGL}(V)$ has prime order, then Theorem \ref{gursax} gives $\nu(x) \geqs 3$ and thus $|x^G| > \frac{1}{2}Qq^{30}$ as in the proof of Proposition  \ref{p:xgbd}. And if $x \in G$ is a field, graph or graph-field automorphism of prime order, then $|x^G|>\frac{1}{16}q^{27}$ (minimal if $x$ is an involutory graph automorphism with $C_G(x)$ of type ${\rm Sp}_8(q)$). Since $|H| < q^8$, we deduce that $\what{Q}(G,H) < 16q^{-11}$ and the result follows.

Next assume $G_0 = {\rm PSp}_{8}(q)$. If $q = 2$, then $G = {\rm Sp}_8(2)$ and $H = S_{10}$ or ${\rm L}_2(17)$ (there is a unique conjugacy class of each type). If $H = S_{10}$ then $b(G,H) = 3$ by the main theorem of \cite{BGS}. And for $H = {\rm L}_2(17)$ it is easy to use {\sc Magma} to check that $\what{Q}(G,H)<1$. Now assume $q \geqs 3$. By inspecting \cite[Table 8.49]{BHR} we see that $q$ is odd. If $H_0 = {\rm L}_2(q^3)$ then we refer the reader to the proof of \cite[Lemma 6.2]{BGS}, where the desired bound $\what{Q}(G,H)<1$ is verified. In the remaining cases, we have $q \geqs 5$, $|H|<q^4$ and $\nu(H) \geqs 3$ by Theorem \ref{gursax}, so $|x^G|>\frac{1}{4}Qq^{15}$ for all $x \in H$ of prime order and the result follows.

Finally, let us assume $G_0 = {\rm P\O}_{8}^{-}(q)$. Here $H_0 = {\rm L}_3^{\e}(q) = {\rm PGL}_3^{\e}(q)$, where $q \equiv -\e \imod{3}$, and we can identify $V$ with the adjoint module for $H_0$. The case $q=2$ can be checked using {\sc Magma} (as in Example \ref{ex:con}(i)), so let us assume $q \geqs 4$. By Theorem \ref{gursax} we have $\nu(H) \geqs 3$, which immediately implies that $\nu(x) \geqs 4$ for all $x \in H_0$ of prime order. Let $W$ be the natural module for $H_0$. 

First assume $x \in H$ is an involution. If $x \in H_0$ and $q$ is odd, then $x$ acts on $V$ as $(-I_4,I_4)$ and thus
\[
|x^G| = \frac{|{\rm SO}_8^{-}(q)|}{2|{\rm SO}_4^{+}(q)||{\rm SO}_4^{-}(q)|} > \frac{1}{4}q^{16} = b_1.
\]
Similarly, if $x \in H_0$ and $q$ is even, then $x$ embeds in $G_0$ as an involution of type $c_4$ (in the notation of \cite{AS}) and we get
\[
|x^G| = \frac{|\Omega_8^{-}(q)|}{q^9|{\rm Sp}_2(q)|} > \frac{1}{4}q^{16} = b_1
\]
once again. In addition, we note that $i_2(H_0) < 2q^4=a_1$ for all $q$. If $x$ is a graph automorphism of $H_0$, then $\nu(x) = 3$, $|x^G|>\frac{1}{4}q^{15} = b_2$ and we observe that there are no more than
\[
\frac{|{\rm PGL}_3^{\e}(q)|}{|{\rm Sp}_2(q)|} < 2q^5 = a_2
\]
such elements in $H$. Note that if $\e=+$ then $q \equiv 2 \imod{3}$ and $q$ is not a square, so $x$ is not an involutory field or graph-field automorphism of $H_0$. We conclude that the total contribution to $\what{Q}(G,H)$ from involutions is less than $a_1^2b_1^{-1} + a_2^2b_2^{-1}$.

Next suppose $q$ is odd and $x \in H_0$ has order $p$. Here we calculate that the elements in $H_0$ with Jordan form $(J_2,J_1)$ on $W$ act on $V$ as $(J_3,J_2^2,J_1)$, whereas the elements  with Jordan form $(J_3)$ act as $(J_5,J_3)$. There are fewer than $q^6 = a_3$ elements of order $p$ in $H_0$, and by considering the Jordan form of $x$ on $V$ we deduce that 
\[
|x^G| \geqs \frac{|{\rm SO}_8^{-}(q)|}{2q^9|{\rm Sp}_2(q)|} > \frac{1}{4}q^{16} = b_3 
\]
(minimal if $x$ has Jordan form $(J_3,J_2^2,J_1)$; see \cite[Lemma 3.18]{FPR2}).

Now assume $x \in H_0$ is a semisimple element of odd prime order. Since the $1$-eigenspace of $x$ on $V$ is $m$-dimensional with $m \in \{2,4\}$ (this follows from the identification of $V$ with the adjoint module), we deduce that 
\[
|x^G| \geqs \frac{|{\rm SO}_8^{-}(q)|}{|{\rm SO}_2^{-}(q)||{\rm GU}_3(q)|}>\frac{1}{2}Qq^{18} = b_4.
\] 
The same lower bound $|x^G| > \frac{1}{2}Qq^{18}$ holds if $x$ is an odd order field automorphism and we note that there are fewer than $\log_2q.q^8 = a_4$ odd order elements in $H$. 

By combining the above estimates, we deduce that \eqref{e:bdnew} holds for all $q \geqs 4$ and the result follows.
\end{proof}

\begin{lem}\label{l:t2_82}
the conclusion of Theorem \ref{t:hat} holds if $G_0 = {\rm P\O}_8^{+}(q)$.
\end{lem}

\begin{proof}
First recall that our definition of the collection $\mathcal{S}$ excludes any subgroups with socle ${\rm P\O}_8^{-}(q_0)$ (where $q=q_0^2$), $\O_7(q)$ (for $q$ odd) or ${\rm Sp}_6(q)$ (for $q$ even). See Remark \ref{r:S}(ii).

If $q=2$ then $H_0 = A_9$ and $b(G,H) = 4$ by \cite[Theorem 1]{BGS}. Similarly, if $q = 3$ then $H_0 = \O_8^{+}(2)$ and $b(G,H) \geqs 3$. For $q=4$ we have $H_0 = {\rm L}_3(4)$ and the proof of \cite[Lemma 6.5]{BGS} gives $\what{Q}(G,H)<1$. So for the  remainder, we may assume $q \geqs 5$. 

The case $q = 5$ requires special attention because there are several possibilities for $H_0$, namely   
\begin{equation}\label{e:o85}
\O_8^{+}(2), \, {\rm U}_3(5), \, {}^2B_2(8), \, A_{10},
\end{equation}
and we need to consider each one in turn.

Suppose $H_0 = {}^2B_2(8)$, in which case $|H_0|<q^7=a_1$. If $x \in H_0$ has prime order, then by considering the action of $\widehat{H}_0 = 2.{}^2B_2(8)$ on $V$, noting that we can use the {\sc Magma} function \texttt{ClassicalMaximals} to construct $\what{H}_0$ as a subgroup of $\O_8^{+}(5)$, we deduce that $|x^G|>\frac{1}{4}q^{16} = b_1$ (minimal if $x$ is an involution, which acts on $V$ as $(-I_4,I_4)$). In addition, if $x \in G$ is a triality graph automorphism then
\[
|x^G| \geqs \frac{|G_0|}{|G_2(q)|} = \frac{1}{4}q^6(q^4-1)^2 = b_2
\]
and we note that $i_3(H \setminus H_0) < q^5 = a_2$. This implies that $\what{Q}(G,H)< a_1^2b_1^{-1}+a_2^2b_2^{-1} < 1$ as required.

Next suppose $H_0 = {\rm U}_3(5)$, so $H \cap G_0 = {\rm PGU}_3(5)$. We can use \texttt{ClassicalMaximals} to construct $2 \times {\rm PGU}_3(5)$ as a subgroup of $\O_8^{+}(5)$, which allows us to determine the action of each prime order element in $H \cap G_0$ on $V$.

Let $x \in H \cap G_0$ be an element of prime order $r$, so $r \in \{2,3,5,7\}$. If $r=2$ then $x$ acts on $V$ as $(-I_4,I_4)$, so $|x^G|>\frac{1}{4}q^{16} = b_1$ and we note that $i_2({\rm PGU}_3(5)) = 525=a_1$. 
Similarly, we have $i_5({\rm PGU}_3(5)) = 15624=a_2$ and we find that $|x^G|>\frac{1}{4}Qq^{16} = b_2$ when $r=5$ (minimal if $x$ has Jordan form $(J_2,J_1)$ on the natural module for ${\rm PGU}_3(5)$, which acts on $V$ as $(J_3,J_2^2,J_1)$). And for $r \in \{3,7\}$ we see that $|x^G|>\frac{1}{2}q^{18} = b_3$ and we note that 
\[
i_3({\rm PGU}_3(5)) + i_7({\rm PGU}_3(5)) = 52550=a_3.
\]
Next suppose $x \in H$ is an involutory graph automorphism of $H_0$. There are at most $|{\rm PGU}_3(5):{\rm PSO}_3(5)| = 3150 = a_4$ such elements in $H$ and they act on $V$ as $(-I_5,I_3)$, so $|x^G|>\frac{1}{4}q^{15} = b_4$. Finally, suppose $x \in G$ is a triality graph automorphism and note that $x^G \cap H$ is contained in $3 \times {\rm PGU}_3(5)$ since we may identify $H \cap G_0 = {\rm PGU}_3(5)$ with the centraliser in $G_0$ of a suitable triality graph automorphism. There are at most 
\[
2(1+i_3({\rm PGU}_3(5))) = 33102 = a_5
\]
such elements in $H$, and as above we have 
$|x^G| \geqs \frac{1}{4}q^6(q^4-1)^2 = b_5$. 

Putting all of this together, we deduce that
\[
\what{Q}(G,H) < \sum_{i=1}^5 a_i^2b_i^{-1} < 1
\]
as required.

Now assume $H_0 = A_{10}$, so either $G = G_0$ and $H = A_{10}$, or $G = G_0.2$ (extended by an involutory graph automorphism) and $H = S_{10}$. If $G = G_0$ then there are $12$ conjugacy classes of $A_{10}$ subgroups and using {\sc Magma} (as in Example \ref{ex:con}(i)) we compute
\[
\what{Q}(G,H) = \frac{7392299}{2046484375}.
\]
Similarly, if $G = G_0.2$ and $H = S_{10}$ (so $H \in \mathcal{A}$) then we can construct $G$ and $H$ as follows:
{\small
\begin{verbatim}
L:=NormalSubgroups(GOPlus(8,5):IndexEqual:=2)[3]`subgroup;
f,G,K:=PermutationRepresentation(L:ModScalars:=true);
H:=f(Normaliser(L,ClassicalMaximals("O+",8,5:classes:={9})[5]));
\end{verbatim}}
\noindent noting that we can replace ``$[5]$" in the final line by ``$[6]$" to get a representative of the second conjugacy class of maximal subgroups in $\mathcal{S}$ that are isomorphic to $S_{10}$. In this way, we calculate that 
\[
\what{Q}(G,H) = \frac{231629473}{4092968750}
\]
and so in both cases we conclude that $\what{Q}(G,H)<1$.

To complete the proof for $q=5$, we may assume $H_0 = \O_{8}^{+}(2)$. Here $\pi(H) = \{2,3,5,7\}$ and we have $G = G_0.X$ and $H = H_0.X$ with $X \leqs S_3$. The main theorem of \cite{BGS} gives $b(G,H) = 2$, but the argument presented in the proof of \cite[Lemma 5.5]{BGS} relies on detailed computations due to T. Breuer \cite{Breuer} in the cases where $G = G_0.3$ or $G_0.S_3$. 

First assume $G = G_0$ or $G_0.2$. With the aid of {\sc Magma}, we compute
\[
\what{Q}(G,H) = \frac{1005545081}{2046484375}
\]
when $G = G_0$, and we get 
\[
\what{Q}(G,H) = \frac{1945445081}{2046484375}
\]
for $G = G_0.2$, so in both cases we have $\what{Q}(G,H)<1$.

Next assume $G = G_0.3$. We claim that 
\[
\what{Q}(G,H) = \frac{1676441081}{2046484375}.
\]
To see this, we first use {\sc Magma} (as above) to determine how each $H_0$-class of prime order elements in $H_0$ embeds in $G_0$, and then we appeal to \cite[Proposition 3.55]{FPR2} to study the action of a triality graph automorphism on this collection of $G_0$-classes. This allows us to calculate that
\begin{equation}\label{e:inner}
\sum_{i=1}^7 a_i^2b_i^{-1} = \frac{1005545081}{2046484375}
\end{equation}
is the precise contribution to $\what{Q}(G,H)$ from the prime order elements in $G_0$, where the $a_i$ and $b_i$ terms are defined in Table \ref{tab:aibi}.

\begin{table}
\[
\begin{array}{cclll} \hline
i & r & x & a_i & b_i \\ \hline
1 & 2 & \texttt{2A}, \texttt{2E} & 1575+ 56700 & 42976171875 \\
2 & & \texttt{2B}, \texttt{2C}, \texttt{2D} & 3{\cdot}3780 & 3{\cdot}153562500 \\
3 & 3 & \texttt{3A}, \texttt{3B}, \texttt{3C} & 3{\cdot}2240 & 3{\cdot}201500000 \\  
4 &  &  \texttt{3D} & 89600 & 2619500000000 \\
5 &  & \texttt{3E} & 268800 & 3438093750000 \\
6 &  5 & \texttt{5A}, \texttt{5B}, \texttt{5C} & 3{\cdot}580608 & 
3{\cdot}47528208000000 \\
7 &  7 & \texttt{7A} & 24883200 & 47151000000000000 \\ \hline
 \end{array}
 \]
 \caption{The case $G = {\rm P\O}_8^{+}(5).3$, $H = \O_8^{+}(2).3$}
 \label{tab:aibi}
 \end{table}
 
We can now complete the calculation by computing the contribution from the triality graph automorphisms in $G$. First observe that $G$ contains four conjugacy classes of triality graph automorphisms, two with $C_{G_0}(x) = G_2(5)$ and two with $C_{G_0}(x) = {\rm PGU}_3(5)$. By arguing as in the proof of \cite[Lemma 2.14]{FPR4}, we calculate that the contribution to $\what{Q}(G,H)$ from triality graph automorphisms is exactly
\begin{equation}\label{e:tri}
2(a_8^2b_8^{-1}+a_9^2b_9^{-1}) = \frac{5367168}{16371875},
\end{equation}
where
\[
a_8 = \frac{|\O_{8}^{+}(2)|}{|G_2(2)|},\; b_8 = \frac{|G_0|}{|G_2(5)|},\;\; a_9 = \frac{|\O_{8}^{+}(2)|}{|{\rm PGU}_3(2)|},\; b_9 = \frac{|G_0|}{|{\rm PGU}_3(5)|},
\]
and the claim follows since 
\[
\frac{1005545081}{2046484375} + \frac{5367168}{16371875} = \frac{1676441081}{2046484375}.
\]

Finally, let us assume $G = G_0.S_3$. Here we claim that 
\[
\what{Q}(G,H) = \frac{4496141081}{2046484375} > 1
\]
and so this case is recorded in Table \ref{tab:new}. To see this, first note that \eqref{e:inner} gives the contribution from the elements of prime order in $G_0$. Now the contribution from the involutions in $G \setminus G_0$ is given by 
\[
a_{10}^2b_{10}^{-1} + a_{11}^2b_{11}^{-1} = \frac{69408}{50375}, 
\]
where
\[
a_{10} = 360,\;\; b_{10} = 3\cdot \frac{|{\rm SO}_{8}^{+}(5)|}{2|{\rm SO}_7(5)|} = 117000
\]
and
\[
a_{11} = 113400,\;\; b_{11} = 3\cdot \frac{|{\rm SO}_{8}^{+}(5)|}{2|{\rm SO}_3(5)||{\rm SO}_5(5)|} = 47604375000.
\]
Similarly, the contribution from elements of order $3$ in $G \setminus G_0$ is given by the expression in \eqref{e:tri}. Therefore,
\[
\what{Q}(G,H) = \frac{1005545081}{2046484375} + \frac{69408}{50375} + \frac{5367168}{16371875} = \frac{4496141081}{2046484375}
\]
as claimed.

For the remainder, we may assume $q \geqs 7$. If $H_0 = \O_8^{+}(2)$ then we refer the reader to the proof of \cite[Lemma 5.5]{BGS}. Similarly, the case $H_0 = {}^3D_4(q_0)$ with $q = q_0^3$ is handled in the proof of \cite[Lemma 5.3]{BGS}. Therefore, to complete the proof we may assume $H_0 = {\rm L}_3^{\e}(q)$, where $q \equiv \e \imod{3}$. Here $H \cap G_0 = {\rm PGL}_3^{\e}(q)$ and this case is labelled $(\mathcal{D}1)$ in \cite[Table 11]{BGS}. The proof of \cite[Lemma 6.5]{BGS} asserts that $\widehat{Q}(G,H)<1$ and details are given for $q$ even. For the sake of completeness, here we give an argument for $q$ odd, which relies on the fact that we can identify $V$ with the adjoint module for $\widehat{H}_0 = {\rm SL}_3^{\e}(q)$.

If $x \in H_0$ is an involution, then $|x^G|>\frac{1}{4}Qq^{16} = b_1$ since $x$ acts on $V$ as $(-I_4,I_4)$ and we note that $i_2(H_0) < 2q^4 = a_1$. If $x \in H$ is an involutory graph automorphism of $H_0$, then $x$ embeds in $G$ as an involution of type $(-I_5,I_3)$, so $|x^G|>\frac{1}{4}q^{15} = b_2$ and there are no more than $2q^5 = a_2$ such elements in $H$. To complete the analysis of involutions, suppose $q = q_0^2$ (so $\e = +$) and $x \in H$ is an involutory field or graph-field automorphism of $H_0$. There are at most
\[
\frac{|{\rm PGL}_3(q)|}{|{\rm PGL}_3(q_0)|} + \frac{|{\rm PGL}_3(q)|}{|{\rm PGU}_3(q_0)|} = 2(q+1)q^3 = a_3
\]
such elements in $H$ and we have $|x^G|>\frac{1}{4}q^{14} = b_3$ since $x$ also acts as an involutory field or graph-field automorphism of $G_0$. Next assume $x \in H_0$ has order $p$. Then by arguing as in the proof of the previous lemma (for the case $G_0 = {\rm P\O}_8^{-}(q)$) we see that $|x^G|>\frac{1}{4}q^{16} = b_4$ and there are fewer than $q^6=a_4$ such elements in $H$. Similarly, if $x \in H$ is semisimple of odd prime order, or if $x$ is a field automorphism of odd order, then $|x^G|>\frac{1}{2}Qq^{18} = b_5$ and in total there are at most $\log_3q.q^8 = a_5$ such elements in $H$. 

To complete the analysis of this case, we may assume $x \in G$ is a triality graph or graph-field automorphism of $G_0$. If $x$ is a graph automorphism, then $|x^G| \geqs \frac{1}{4}q^6(q^4-1)^2 = b_6$ and we note that there are at most 
\[
2(1+i_3({\rm PGL}_3^{\e}(q))) < 4(q+1)q^5 = a_6
\]
such elements in $H$. And by arguing as in the proof of \cite[Lemma 6.5]{BGS}, if $q = q_0^3$ and $x \in G$ is a triality graph-field automorphism, then $|x^G|>\frac{1}{4}q^{56/3} = b_7$ and there are at most 
\[
2\cdot \frac{|{\rm PGL}_3^{\e}(q)|}{|{\rm PGL}_3^{\e}(q_0)|} < 4q^{16/3} = a_7
\]
such elements in $H$. 

Bringing these estimates together, we deduce that
\[
\what{Q}(G,H) < \sum_{i=1}^7 a_i^2b_i^{-1},
\]
which is less than $1$ for $q \geqs 11$. Finally, if $q=7$ then $i_3({\rm PGL}_3(q)) = 123578$ and by setting $a_6 = 2(1+123578) = 247158$ in the previous bound we conclude that $\what{Q}(G,H)<1$, as required.
\end{proof}

\begin{lem}\label{l:t2_7}
the conclusion of Theorem \ref{t:hat} holds if $n=7$. 
\end{lem}

\begin{proof}
Suppose $G_0 = {\rm L}_{7}^{\e}(q)$. Here $H_0 = {\rm U}_3(3)$ with $q = p \equiv \e \imod{4}$ and $q \geqs 5$. Then $|H|<q^6$ and $\nu(H) = 2$, so $|x^G|>\frac{1}{2}Qq^{20}$ for all $x \in H$ of prime order and we conclude that $\what{Q}(G,H) < 2Q^{-1}q^{-8}<1$.

Now assume $G_0 = \O_{7}(q)$. If $H_0 = G_2(q)$, then $b(G,H) = 4$ by the main theorem of \cite{BGS}. Similarly, if $q=3$ then the main theorem of \cite{BGS} gives $b(G,H) \geqs 3$ for all $H \in \mathcal{S}$, so we may assume $q =p \geqs 5$, $G = G_0$ and $H = {\rm Sp}_6(2)$. For $q = 5$, we can use the function \texttt{QHat} (see Example \ref{ex:qhat}) to compute
\[
\what{Q}(G,H) = \frac{29656873}{31484375},
\]
and similarly we get
\[
\what{Q}(G,H) = \frac{57275}{722701}
\]
when $q=7$. 

Now assume $q \geqs 11$, so every prime order element in $H$ is semisimple when viewed as an element of $G$. In particular, this means that we can inspect the character table of $H$ (see \cite{GAPCTL}, for example) to determine the embedding of each $H$-class of elements of prime order. For example, if $x \in H$ is an element in the \texttt{3A} conjugacy class of $H$ and $\chi$ is the character of the relevant representation $\rho:H \to G$, then we read off $|x^H| = 672$ and $\chi(x) = 4$. This means that $x$ acts on $\bar{V} = V \otimes k$ as $(I_5,\omega,\omega^2)$, up to conjugacy, where $k$ is the algebraic closure of $\mathbb{F}_q$ and $\omega \in k$ is a primitive cube root of unity, and this implies that
\[
|x^G| \geqs \frac{|{\rm SO}_7(q)|}{|{\rm SO}_5(q)||{\rm GU}_1(q)|}>\frac{1}{2}Qq^{10}.
\]

\renewcommand{\arraystretch}{1.1}
\begin{table}
\[
\begin{array}{ccllll} \hline
i & r & x & \bar{V}{\downarrow}\,x & a_i & b_i \\ \hline
1 & 2 & \texttt{2A} & (-I_6,I_1) & 63 & \frac{1}{4}q^{6} \\
2 &  & \texttt{2B}, \texttt{2D} &  (-I_4,I_3) & 315+3780 & \frac{1}{4}q^{12} \\
3 &  & \texttt{2C} & (-I_2,I_5) & 945 & \frac{1}{4}Qq^{10} \\
4 & 3 & \texttt{3A} & (I_5,\omega,\omega^2) &  672 & \frac{1}{2}Qq^{10} \\
5 &  & \texttt{3B} & (I_1,\omega I_3,\omega^2I_3) &  2240 & \frac{1}{2}Qq^{12} \\
6 & & \texttt{3C} & (I_3,\omega I_2,\omega^2I_2) &  13440 & \frac{1}{2}Qq^{14} \\
7 & 5 & \texttt{5A} & (I_3,\omega, \ldots, \omega^4) & 48384 & \frac{1}{2}Q^2q^{16} \\ 
8 & 7 & \texttt{7A} & (I_1,\omega, \ldots, \omega^6) & 207360 & \frac{1}{2}Q^3q^{18} \\ \hline
\end{array}
\]
\caption{The case $G = \O_7(q)$, $H = {\rm Sp}_6(2)$, $q =p \geqs 11$}
\label{tab:sp62}
\end{table}
\renewcommand{\arraystretch}{1}

In this way, we deduce that 
\[
\what{Q}(G,H) < \sum_{i=1}^8 a_i^2b_i^{-1} < 1
\]
for all $q \geqs 11$, where the $a_i$ and $b_i$ terms are defined in Table \ref{tab:sp62} (in each case, $\omega \in k$ is a primitive $r$-th root of unity).
\end{proof}

\begin{lem}\label{l:t2_6}
the conclusion of Theorem \ref{t:hat} holds if $n=6$. 
\end{lem}

\begin{proof}
To begin with, let us assume $G_0 = {\rm L}_{6}^{\e}(q)$. If $q=2$ then $\e=-$, $H_0 = {\rm U}_4(3)$ or ${\rm M}_{22}$ and the main theorem of \cite{BGS} gives $b(G,H) \geqs 3$. Now assume $q \geqs 3$. If $H_0 = {\rm L}_3^{\e}(q)$ then we are in case $(\mathcal{C}1)$ from \cite[Table 5]{BGS} and we refer the reader to the proof of \cite[Lemma 5.2]{BGS}. Similarly, the case $H_0 = {\rm U}_4(3)$ is labelled $(\mathcal{C}12)$ in \cite[Table 5]{BGS} and the details are given in the proof of \cite[Lemma 5.4]{BGS}. To complete the argument for linear and unitary groups, we may assume $q \geqs 5$ and $H_0 = A_6$, $A_7$, ${\rm L}_2(11)$ or ${\rm L}_3(4)$. 

First we claim that $\nu(H) \geqs 2$ in each of these cases. This follows from Theorem \ref{gursax} for $H_0 = {\rm L}_2(11)$, whereas the three remaining cases are recorded in Table \ref{ctab}. By the main theorem of \cite{Kan}, none of these subgroups contain unipotent elements with Jordan form $(J_2,J_1^4)$ on $V$. And by inspecting the relevant character tables, one can also check that $\nu(x) \geqs 2$ for every element $x \in H \cap {\rm PGL}(V)$ of prime order $r \ne p$. This justifies the claim.

Next observe that $i_2(H) \leqs 1963 = a_1$ and note that the lower bound $\nu(H) \geqs 2$ implies that $|x^G|>\frac{1}{2d}q^{14} = b_1$ for every involution $x \in H$, where $d = (6,q-\e)$ (minimal if $x \in G$ is an involutory graph automorphism with $C_G(x)$ of type ${\rm Sp}_6(q)$). In addition, we note that $H$ contains at most $a_2 = 18656$ elements of odd prime order and we have $|x^G|>\frac{1}{2}Qq^{15} = b_2$ for all such elements. These estimates imply that $\what{Q}(G,H) < a_1^2b_1^{-1} + a_2^2b_2^{-1} < 1$ for all $q \geqs 5$, as required.

To complete the proof, we may assume $G_0 = {\rm PSp}_6(q)$.  If $H_0$ is a group of Lie type in characteristic $p$, then either $H_0 = {\rm L}_2(q)$ and $p \geqs 7$, or $H_0 = G_2(q)$ and $q \geqs 4$ is even. In the latter case, the main theorem of \cite{BGS} gives $b(G,H) = 4$, while the proof of \cite[Lemma 6.2]{BGS} shows that $\what{Q}(G,H)<1$ in the former case. 

For the remainder, we may assume $H_0$ is not a group of Lie type in characteristic $p$. If $q = 2$ then $H_0 = {\rm U}_3(3)$ is the only possibility and \cite[Theorem 1]{BGS} states that $b(G,H) = 4$. The cases $q \in \{3,5\}$ can be handled using {\sc Magma}. For example, if $q =5$ and $H_0 = {\rm J}_2$, then we compute
\[
\what{Q}(G,H) \leqs \frac{484897}{1259375},
\]
with equality if $G = G_0.2$. Finally, suppose $q \geqs 7$ is odd (note that if $q \geqs 4$ is even, then $H_0 = G_2(q)$ is the only possibility). By inspecting \cite[Table 8.29]{BHR} we see that $H_0$ is one of the following:
\[
A_5, \, A_7, \, {\rm L}_2(7), \, {\rm L}_2(13), \, {\rm U}_3(3), \, {\rm J}_2.
\]

In every case, we claim that $\nu(H) \geqs 2$. This follows immediately from Theorem \ref{gursax} unless $H_0 = {\rm U}_3(3)$ or ${\rm J}_2$. For $H_0 = {\rm U}_3(3)$ the claim is clear unless $q = 7$ (since we have $\nu(x) \geqs 2$ for every nontrivial semisimple element $x \in G$), in which case we can use {\sc Magma} to check that $\nu(x) = 5$ for every element $x \in H_0$ of order $7$. Similarly, if $H_0 = {\rm J}_2$ then we may assume $q = 9$ and it is easy to check that $\nu(x) \geqs 3$ for all $x \in H_0$ of order $3$. This justifies the claim. 

Suppose $H_0 \ne {\rm J}_2$. Then $i_2(H) \leqs 315 = a_1$ and the previous claim implies that $|x^G|>\frac{1}{2}q^8 = b_1$ for every involution $x \in H$. In addition, $H$ contains at most $a_2 = 2456$ elements of odd prime order and we note that $|x^G|>\frac{1}{4}Qq^{10} = b_2$ for all such elements. These estimates imply that  $\what{Q}(G,H) < a_1^2b_1^{-1} + a_2^2b_2^{-1} < 1$ for all $q \geqs 7$.

Finally, let us assume $H_0 = {\rm J}_2$. By inspecting \cite[Table 8.29]{BHR}, we may assume $q \geqs 9$. If $x \in H_0$ is an involution, then either $x \in \texttt{2A}$ and we have $\nu(x) = 2$ and $|x^G|>\frac{1}{2}q^8 = b_1$, or $x \in \texttt{2B}$ and $\nu(x) = 3$, which implies that $|x^G|>\frac{1}{4}Qq^{12} = b_2$. Therefore, the contribution to $\what{Q}(G,H)$ from the involutions in $H_0$ is less than $a_1^2b_1^{-1}+a_2^2b_2^{-1}$, where $a_1 = 315$ and $a_2 = 2520$. And if $x \in H \setminus H_0$ is an involution, then $q = p^2$ and $x$ is an involutory field automorphism of $G_0$ (see \cite[Table 8.29]{BHR}). So the contribution from the latter involutions is at most $a_3^2b_3^{-1}$, where $a_3 = 1800$ and $b_3 = \frac{1}{4}q^{21/2}$. Finally, let us assume that $x \in H_0$ has odd prime order. Here $\nu(x) \geqs 3$, with equality if and only if $x$ is in the $H_0$-class labelled \texttt{3A}. In the latter case, we have $|x^G|>\frac{1}{2}Qq^{12} = b_4$ and we note that the \texttt{3A} class in $H_0$ has size $a_4 = 560$. For all other elements $x \in H_0$ of odd prime order, we have $\nu(x) \geqs 4$ and thus $|x^G|>\frac{1}{2}Qq^{14} = b_5$.  Since there are $a_5 = 131984$ such elements in $H_0$, we conclude that
\[
\what{Q}(G,H) < \sum_{i=1}^5 a_i^2b_i^{-1} < 1
\]
for all $q \geqs 9$.
\end{proof}

\begin{lem}\label{l:t2_5}
the conclusion of Theorem \ref{t:hat} holds if $n=5$. 
\end{lem}

\begin{proof}
Here $G_0 = {\rm L}_5^{\e}(q)$ and $q = p$ (see \cite[Tables 8.19, 8.21]{BHR}). If $q = 2$ then $\e = -$, $H_0 = {\rm L}_2(11)$ and the main theorem of \cite{BGS} gives $b(G,H) = 2$. However, using {\sc Magma} we calculate that $\what{Q}(G,H) = \frac{799}{576}$ if $G = G_0$ and $\what{Q}(G,H) = \frac{931}{576}$ if $G = G_0.2$, so these special cases are recorded in Table \ref{tab:new}. For instance, if $G = G_0.2$ then we can construct $G$ and $H$ as in  Example \ref{ex:con}(iv) and we can use the function \texttt{QHat} (see Example \ref{ex:qhat}) to compute $\what{Q}(G,H)$ precisely.

In the same way, we can use {\sc Magma} to check that $\what{Q}(G,H)<1$ when $q \in \{3,5,7\}$, so for the remainder we may assume $q \geqs 11$ and $H_0 = {\rm L}_2(11)$ or ${\rm U}_4(2)$. If $H_0 = {\rm L}_2(11)$ then $q \geqs 13$, $|H|<q^3$ and by inspecting the character table of $H$ we see that $\nu(H) = 2$, so $|x^G|>\frac{1}{2}q^{12}$ for all $x \in H$ of prime order and thus $\what{Q}(G,H) < 2q^{-6}<1$.

Finally, suppose $q \geqs 11$ and $H_0 = {\rm U}_4(2)$. If $x \in H_0$ is an involution in the $H_0$-class labelled \texttt{2A}, then $x$ acts as $(-I_4,I_1)$ on $V$, so $|x^G|>\frac{1}{2}q^{8}=b_1$ and we note that the \texttt{2A} class has size $a_1 = 45$. For all other elements $x \in H$ of prime order we have $|x^G|>\frac{1}{2}Qq^{12} = b_2$ and we observe  that $|H|<4q^4 = a_2$. In the usual manner, these estimates imply that 
\[
\what{Q}(G,H) < a_1^2b_1^{-1} + a_2^2b_2^{-1}<1
\]
and the result follows.
\end{proof}

\begin{lem}\label{l:t2_4}
the conclusion of Theorem \ref{t:hat} holds if $n=4$. 
\end{lem}

\begin{proof}
First assume $G_0 = {\rm L}_4(q)$. If $q = 2$ then $H_0 = A_7$ and we exclude this case since $G_0 \cong A_8$ (see Remark \ref{r:S2}(i)). In the remaining cases, we have $q = p \geqs 7$ and $H_0 = A_7$, ${\rm L}_2(7)$ or ${\rm U}_4(2)$, so $\pi(H) \subseteq \{2,3,5,7\}$. If $q = 7$ then $H_0 = {\rm U}_4(2)$ and we can use {\sc Magma} to verify the bound $\what{Q}(G,H)<1$. For example, if $G = G_0.2$ (extended by a symplectic-type graph automorphism), then $H = {\rm U}_4(2).2$ and we can construct $G$ and $H$ as follows:
{\small
\begin{verbatim}
G:=LowIndexSubgroups(AutomorphismGroupSimpleGroup("L",4,7),2)[3];
H:=MaximalSubgroups(G)[2]`subgroup;
\end{verbatim}}
\noindent This allows us to use the function \texttt{QHat} (see Example \ref{ex:qhat}) to show that  
\[
\what{Q}(G,H) = \frac{522449}{2235331}.
\]
We can also handle the cases $q \in \{11,13\}$ using {\sc Magma}, so for the remainder we may assume $q \geqs 19$ (note that $\mathcal{S}$ is empty when $q = 17$).

First observe that $i_r(H) \leqs a_r$, where
\[
a_2 = 891,\; a_3 = 800,\;  a_5 = 5184,\; a_7 = 720.
\]
More precisely, we have $i_2(H_0) \leqs 315 = a_{2,1}$ and $i_2(H \setminus H_0) \leqs 576 = a_{2,2}$. Let $x \in G$ be an element of prime order $r$. If $x$ is not an involutory graph automorphism, then  
\[
|x^G| \geqs \frac{|{\rm GL}_4(q)|}{q^5|{\rm GL}_2(q)||{\rm GL}_1(q)|} = (q^3-1)(q^2+1)(q-1) = b
\]
(minimal if $x$ is unipotent with Jordan form $(J_2,J_1^2)$), otherwise  
\[
|x^G| \geqs \frac{|G_0|}{|{\rm Sp}_4(q)|} = \frac{1}{d}q^2(q^3-1) = c,
\]
where $d = (4,q-1)$. It follows that
\begin{equation}\label{e:bdd}
\what{Q}(G,H) \leqs (a_{2,1}^2+a_3^2+a_5^2+a_7^2)b^{-1} + a_{2,2}^2c^{-1}<1
\end{equation}
for all $q \geqs 19$. 

Next assume $G_0 = {\rm U}_4(q)$. Here $q = p \geqs 3$. If $q = 3$ then $H_0 = A_7$ or ${\rm L}_3(4)$ and the main theorem of \cite{BGS} gives $b(G,H) \geqs 3$. Next assume $q = 5$. If $G = G_0$ then $H = A_7$ or ${\rm U}_4(2)$ and we use {\sc Magma} to show that $\what{Q}(G,H)<1$.  For example, if $H = {\rm U}_4(2)$ then we compute
\[
\what{Q}(G,H) = \frac{11341}{11375}.
\]
Next assume $G = G_0.2$ (extended by a graph automorphism), so $H = S_7$ or ${\rm U}_4(2).2$. If $H = S_7$ then we get $\what{Q}(G,H)<1$. On the other hand, if $H = {\rm U}_4(2).2$ then $b(G,H)=2$ by \cite[Theorem 1]{BGS}, but we compute
\[
\what{Q}(G,H) = \frac{23941}{11375}
\]
and so this case is recorded in Table \ref{tab:new}. Note that in the latter case, we can construct $G$ and $H$ as follows:
{\small
\begin{verbatim}
G:=LowIndexSubgroups(AutomorphismGroupSimpleGroup("U",4,5),2)[3];
H:=MaximalSubgroups(G)[6]`subgroup;
\end{verbatim}}

Now assume $q \geqs 11$. By arguing as above (for $G_0 = {\rm L}_4(q)$), we see that the upper bound in \eqref{e:bdd} holds, where the terms $a_{2,1}$, $a_{2,2}$, $a_3$, $a_5$ and $a_7$ are defined as before, and we have
\[
b = \frac{|{\rm GU}_4(q)|}{|{\rm GU}_3(q)||{\rm GU}_1(q)|} = q^3(q^2+1)(q-1)
\]
and $c = \frac{1}{d}q^2(q^3+1)$, where $d = (4,q+1)$. It is routine to check that this upper bound is sufficient for all $q \geqs 23$. And if $H_0 = A_7$ then by setting 
\[
a_{2,1} = 105,\; a_{2,2} = 126,\; a_3 = 350,\; a_5 = 504, \; a_7 = 720,
\]
one can check that the upper bound in \eqref{e:bdd} is good enough for all $q \geqs 11$.

So to complete the proof, we may assume $H_0 = {\rm U}_4(2)$ and $q \in \{11,17\}$. By inspecting the character table of $\widehat{H}_0 = 2.{\rm U}_4(2)$, we can determine how each conjugacy class in $H_0$ embeds in $G_0$. In this way, for $G = G_0$ we deduce that 
\[
\what{Q}(G,H) \leqs |{\rm GU}_4(q)|^{-1}\sum_{i=1}^6 a_i^2c_i < 1,
\]
where the $a_i$ and $c_i$ terms are defined as in Table \ref{tab:u4q} (in the table, we set $\bar{V} = V \otimes k$, where $k$ is the algebraic closure of $\mathbb{F}_q$, and we write $\omega$ for a primitive $r$-th root of unity in $k$).

\renewcommand{\arraystretch}{1.1}

\begin{table}
\[
\begin{array}{ccllll} \hline
i & r & x & \bar{V}{\downarrow}\,x & a_i &  c_i \\ \hline
1 & 2 & \texttt{2A},\, \texttt{2B} & (-I_2,I_2) & 45+270 & 2|{\rm GU}_2(q)|^2 \\
2 & 3 & \texttt{3A} & (I_1,\omega I_3) & 40 & |{\rm GU}_1(q)||{\rm GU}_3(q)| \\
3 & & \texttt{3B} & (I_1,\omega^2 I_3) & 40 & |{\rm GU}_1(q)||{\rm GU}_3(q)| \\
4 & & \texttt{3C} & (\omega I_2, \omega^2I_2) & 240 & |{\rm GU}_2(q)|^2 \\
5 & & \texttt{3D} & (I_2,\omega,\omega^2) & 480 & |{\rm GU}_2(q)||{\rm GU}_1(q)|^2 \\
6 & 5 & \texttt{5A} & (\omega,\omega^2,\omega^3,\omega^4) & 5184 & |{\rm GU}_1(q)|^4 \\ \hline 
\end{array}
\]
\caption{The case $G_0 = {\rm U}_4(q)$, $H_0 = {\rm U}_4(2)$, $q \in \{11,17\}$}
\label{tab:u4q}
\end{table}

\renewcommand{\arraystretch}{1}

Now assume  $G = G_0.2$, where $G_0$ is extended by an involutory graph automorphism. As above, it is straightforward to calculate the contribution to $\what{Q}(G,H)$ from the prime order elements in $G_0$, keeping in mind that the \texttt{3A} and \texttt{3B} classes in $H_0 = {\rm U}_4(2)$ are fused in $H = H_0.2$. So it just remains to include the contribution from involutory graph automorphisms, noting that there are two such classes in $H$. By arguing as in the final paragraph of the proof of \cite[Lemma 2.23]{FPR4}, we see that if $x \in H$ is a symplectic-type graph automorphism, then $C_{G_0}(x)$ is also of type ${\rm Sp}_4(q)$. And similarly, if $x$ is a non-symplectic graph automorphism of $H_0$ then $x$ embeds in $G$ as a non-symplectic graph automorphism of $G_0$. Therefore, the contribution to $\what{Q}(G,H)$ from involutory graph automorphisms is at most 
\[
|{\rm U}_4(q)|^{-1} \left(r_1^2s_1 + r_2^2s_2\right) <q^{-1},
\]
where $r_1 = 36$, $r_2 = 540$, $s_1 = |{\rm Sp}_4(q)|$ and $s_2 = |{\rm SO}_4^{-}(q)|$. And by combining this estimate with the contribution from the prime order elements in $G_0$, we conclude that $\what{Q}(G,H)<1$, as required.

Finally, let us assume $G_0 = {\rm PSp}_4(q)$. Suppose $q= 2^f$ is even and recall that we may assume $f \geqs 2$ (see \eqref{e:groups}). In fact, as noted in Remark \ref{r:S}(i), we may assume $f \geqs 3$ is odd and $H_0 = {}^2B_2(q)$. Then the main theorem of \cite{BGS} gives $b(G,H) = 3$.

Now assume $q$ is odd. The possibilities for $H_0$ are listed in \cite[Table 8.13]{BHR} and we see that $q \geqs 5$. If $q = 5$, then $H_0 = A_6$ and using {\sc Magma} we compute $\what{Q}(G,H) = \frac{331}{650}$ when $G = G_0$ and $\what{Q}(G,H) = \frac{841}{650}$ when $G = G_0.2$, so the latter case is included in Table \ref{tab:new} since $b(G,H) = 2$. Similarly, if $q = 7$ and  $H_0 = A_7$ then $\what{Q}(G,H) = \frac{47}{49}$ when $G = G_0$ and $\what{Q}(G,H) = \frac{293}{196}$ when $G = G_0.2$, which means that the latter case appears in Table \ref{tab:new}. And if $G = G_0.2$ and $H = {\rm PGL}_2(7)$, then $\what{Q}(G,H) = \frac{53}{1225}$.

For the remainder, we may assume $q \geqs 9$, so either $q=p$ and $H_0 = A_6$, or $p \geqs 5$ and $H_0 = {\rm L}_2(q)$. In particular, we may assume $q \geqs 11$. The cases $q \in \{11,13\}$ can be checked with {\sc Magma}, so let us assume $q \geqs 17$. For $H_0 = {\rm L}_2(q)$, we refer the reader to the analysis of case $(\mathcal{S}14)$ in the proof of \cite[Proposition 7.1]{BGS}, where the bound $\what{Q}(G,H)<1$ is established. Finally, suppose $H_0 = A_6$. We have $i_2(H) \leqs 75 = a_1$ and $|x^G| \geqs \frac{1}{2}q^2(q^2-1) = b_1$ for every involution $x \in H$. And if $x \in H$ has odd prime order, then $|x^G| \geqs q^3(q-1)(q^2+1) = b_2$ and we note that there are $224=a_2$ such elements in $H$. Therefore,
\[
\what{Q}(G,H) \leqs a_1^2b_1^{-1} + a_2^2b_2^{-1} < 1
\]
for all $q \geqs 17$.
\end{proof}

\begin{lem}\label{l:t2_3}
the conclusion of Theorem \ref{t:hat} holds if $n=3$. 
\end{lem}

\begin{proof}
First assume $G_0 = {\rm L}_3(q)$, in which case $H_0 = {\rm L}_2(7)$ or $A_6$ and $\pi(H) \subseteq \{2,3,5,7\}$. If $q =4$ then $H_0 = A_6$ and the main theorem of \cite{BGS} states that $b(G,H) = 3$. In fact, this is incorrect (as noted in Remark \ref{r:main}(e)) and we claim that $b(G,H) = 3+\a$, where $\a = 1$ if $G = G_0.2^2$ and $H = A_6.2^2$, otherwise $\a=0$. 

To check this, we can easily construct each pair $(G,H)$ in {\sc Magma} and for $G \ne G_0.2^2$ we can use random search to find elements $x,y \in G$ such that $H \cap H^x \cap H^y = 1$; since $|H|^2 > |G|$, this implies that $b(G,H) = 3$. And for $G = G_0.2^2$ we use random search to establish the bound $b(G,H) \leqs 4$ and we can show that equality holds by using the \texttt{RegOrbits} function from \cite[Section 1.2]{AB_comp}:
{\small
\begin{verbatim}
G:=LowIndexSubgroups(AutomorphismGroupSimpleGroup("L",3,4),3)[1];
H:=MaximalSubgroups(G)[5]`subgroup;
RegOrbits(G,[H,H,H]);
\end{verbatim}}

\noindent This returns \texttt{0} and we conclude that the triple $(H,H,H)$ is non-regular, which means that $b(G,H) \geqs 4$.

For the remainder, we may assume $q \geqs 11$ (see \cite[Table 8.4]{BHR}). For $r \in \{2,3,5,7\}$ we compute $i_r(H) \leqs a_r$, where $a_2 = 45$, $a_3 = 80$, $a_5 = 144$ and $a_7 = 48$. If $q=p$ then 
$|x^G| \geqs q^2(q^2+q+1) = b$ for all $x \in G$ of prime order and it follows that 
\[
\widehat{Q}(G,H) \leqs (a_2^2+a_3^2+a_5^2+a_7^2)b^{-1},
\]
which is less than $1$ for all $q \geqs 37$. This leaves the cases $q \in \{11,19,23,29,31\}$, which can be checked directly using {\sc Magma}. Now assume $q = p^2$ with $p \geqs 7$. Since $H \leqs A_6.2^2$ we see that $r \in \{2,3,5\}$ and $i_r(H) \leqs a_r$, where $a_2 = 111$, $a_3 = 80$ and $a_5 = 144$. In addition, we note that
\[
|x^G| \geqs \frac{|G_0|}{|{\rm PGU}_3(q^{1/2})|} = \frac{1}{3}q^{3/2}(q^{3/2}-1)(q+1) = b
\]
(minimal if $x \in G$ is an involutory graph-field automorphism) and thus
\[
\widehat{Q}(G,H) \leqs (a_2^2+a_3^2+a_5^2)b^{-1} < 1
\]
for all $q \geqs 49$.

Now assume $G_0 = {\rm U}_3(q)$, in which case $q=p \geqs 3$. If $q = 3$ then there is a unique class of subgroups in $\mathcal{S}$ (with $H_0 = {\rm L}_2(7)$) and the main theorem of \cite{BGS} gives $b(G,H)= 3+\a$, where $\a = 1$ if $G = G_0.2$, otherwise $\a = 0$. Similarly, if $q = 5$ then $b(G,H) \geqs 3$. Now assume $q \geqs 11$, so $H_0 = {\rm L}_2(7)$ or $A_6$ and $\pi(H) \subseteq \{2,3,5,7\}$ (note that $\mathcal{S}$ is empty if $q = 7$). We calculate that $i_r(H) \leqs a_r$, where
\[
a_2 = 81, \, a_3 = 80, \, a_5 = 144,\, a_7 = 48.
\]
Since $|x^G| \geqs q^2(q^2-q+1)=b$ for all $x \in G$ of prime order, we deduce that
\[
\what{Q}(G,H) \leqs (a_2^2+a_3^2+a_5^2+a_7^2)b^{-1} < 1
\]
for all $q \geqs 17$. Finally, the cases $q \in \{11,13\}$ can be checked using {\sc Magma}. 
\end{proof}

\begin{lem}\label{l:t2_2}
the conclusion of Theorem \ref{t:hat} holds if $n=2$. 
\end{lem}

\begin{proof}
Here $G_0 = {\rm L}_2(q)$, $H_0=A_5$ and we have $q \geqs 9$. In fact, we may assume $q \geqs 11$ since ${\rm L}_2(9) \cong A_6$ (see Remark \ref{r:S2}(ii)). If $q \in \{11,19\}$, then the main theorem of \cite{BGS} gives $b(G,H) = 3$, so we may assume $q \geqs 29$ (see \cite[Table 8.2]{BHR}). For $q = 29$ we have $G = G_0$, $H = H_0$ and we compute $\what{Q}(G,H) = \frac{277}{203}$. Similarly, for $q = 31$ we get $\what{Q}(G,H) = \frac{73}{62}$. In the latter two cases, we have $b(G,H) = 2$ and so these special cases are recorded in Table \ref{tab:new}.

We claim that $\what{Q}(G,H)<1$ for all $q \geqs 41$. If $x \in H$ has odd prime order then $|x^G| \geqs q(q-1) = b_1$ and we note that there are $44=a_1$ such elements in $H$. If $x \in H_0$ is an involution, then $|x^G| \geqs \frac{1}{2}q(q-1) = b_2$ and we have $i_2(H_0) = 15=a_2$. Similarly, if $x \in H \setminus H_0$ is an involution, then $q = q_0^2$, 
\[
|x^G| \geqs \frac{|G_0|}{|{\rm PGL}_2(q^{1/2})|} = \frac{1}{2}q^{1/2}(q+1)=b_3
\]
and we note that $i_2(H\setminus H_0) = 10 = a_3$. In view of Lemma \ref{l:favbound}, these estimates imply that
\[
\what{Q}(G,H) \leqs \sum_{i=1}^3 a_i^2b_i^{-1}
\]
and one checks that this upper bound is less than $1$ for all $q \geqs 71$. This leaves us with the cases $q \in \{41,49,59,61\}$, which we can handle using {\sc Magma}.
\end{proof}

\vs

This completes the proof of Theorem \ref{t:hat}.

\section{Proof of Theorem \ref{t:main}}\label{s:main}

We are now ready to prove Theorem \ref{t:main}. Let $\tau = (H_1,H_2)$ be a pair of subgroups in $\mathcal{S}$. By Lemma \ref{l:key}, we see that $\tau$ is regular if $\what{Q}(G,H_i)<1$ for $i=1,2$. This means that the proof of Theorem \ref{t:main} is immediately reduced to the groups $G$ with a maximal subgroup $H \in \mathcal{S}$ with $\what{Q}(G,H) \geqs 1$. The cases with $b(G,H) \geqs 3$ are recorded in \cite[Table 1]{BGS}, while those with $b(G,H) = 2$ and $\what{Q}(G,H) \geqs 1$ are determined by Theorem \ref{t:hat}. Therefore, we may assume $G$ is one of the groups recorded in \cite[Table 1]{BGS} or Table \ref{tab:new}, and we consider each possibility in turn. As usual, we write $n$ for the  dimension of the natural module for $G_0$ and we continue to assume (as we may) that $G_0$ is one of the groups in \eqref{e:groups}.

\begin{lem}\label{l:n9}
the conclusion of Theorem \ref{t:main} holds if $n \geqs 9$.
\end{lem}

\begin{proof}
By inspecting \cite[Table 1]{BGS} and Table \ref{tab:new}, we may assume that $G_0$ is one of the following:
\[
\O_{14}^{+}(2), \, {\rm Sp}_{12}(2), \, \O_{12}^{-}(2), \, \O_{10}^{-}(2).
\]

First assume $G_0 = \O_{14}^{+}(2)$. Here we can use the {\sc Magma} function \texttt{ClassicalMaximals} to construct a set of representatives of the conjugacy classes of maximal subgroups in $\mathcal{S}$ (see Example \ref{ex:con}(i)). Both $G_0$ and $G_0.2$ have unique classes of subgroups with socle $H_0 = A_{16}$ and ${\rm L}_2(13)$. For $G = G_0$ we can use random search (see Example \ref{ex:random}) to show that every pair of subgroups in $\mathcal{S}$ is regular (that is to say, if $\tau = (H_1,H_2)$ is such a pair, then we can find an element $x \in G$ such that $H_1 \cap H_2^x = 1$). In the same way, we find that every triple of subgroups in $\mathcal{S}$ is regular when $G = G_0.2 = {\rm O}_{14}^{+}(2)$. And every pair is regular apart from $\tau = (H,H)$ with $H = S_{16}$, where the main theorem of \cite{BGS} gives $b(G,H) = 3$. So the latter case is recorded in Table \ref{tab:2}.

For $G = {\rm Sp}_{12}(2)$ we can use {\sc Magma} and random search to check that every pair of subgroups in $\mathcal{S}$ is regular. 

Next assume $G_0 = \O_{12}^{-}(2)$. If $G = G_0$, then there are three classes of subgroups in $\mathcal{S}$; one class of subgroups isomorphic to $A_{13}$, and two classes of subgroups isomorphic to ${\rm L}_3(3).2$. By \cite[Theorem 1]{BGS} we know that $(A_{13},A_{13})$ is non-regular, and with the aid of {\sc Magma}, using random search, it is easy to verify that this is the only non-regular pair. Similarly, $(S_{13},S_{13})$ is the only non-regular pair for $G = {\rm O}_{12}^{-}(2)$. In particular, in both cases we find that all triples are regular.

Finally, suppose $G_0 = \O_{10}^{-}(2)$. If $G = G_0$ then there is a unique class of subgroups in $\mathcal{S}$, which are isomorphic to $A_{12}$. By \cite[Theorem 1]{BGS}, we see that $(H,H,H)$ is regular, but $(H,H)$ is non-regular. Now assume $G = {\rm O}_{10}^{-}(2)$. Here there are two classes of subgroups in $\mathcal{S}$, which are isomorphic to $S_{12}$ and ${\rm M}_{12}.2$, respectively. Let $H$ and $K$ be representatives of these two classes. By appealing to \cite[Theorem 1]{BGS}, $(H,H,H,H)$ is regular, but $(H,H,H)$ and $(H,H)$ are non-regular and so the latter two cases are recorded in Tables \ref{tab:3} and \ref{tab:2}, respectively. In addition, $(K,K)$ is regular since $b(G,K) = 2$, and we note that $(H,K)$ is non-regular since $|H||K|>|G|$. Finally, we can use {\sc Magma}, with random search, to verify that the triple $(H,H,K)$ is regular (see Example \ref{ex:con}(ii) for the construction of $K$ as a subgroup of $G$). 
\end{proof}

\begin{lem}\label{l:n8}
the conclusion of Theorem \ref{t:main} holds if $n = 8$.
\end{lem}

\begin{proof}
Here $G_0 = {\rm Sp}_8(2)$ or ${\rm P\O}_8^{+}(q)$ with $q \in \{2,3,5\}$. For $G = {\rm Sp}_8(2)$ there are two classes of subgroups in $\mathcal{S}$, represented by $H = S_{10}$ and $K = {\rm L}_2(17)$. By \cite[Theorem 1]{BGS}, we see that 
$(H,H,H)$ and $(K,K)$ are regular, whereas $(H,H)$ is non-regular. And using {\sc Magma}, it is easy to check that $(H,K)$ is regular.

Next assume $G_0 = \O_8^{+}(2)$, in which case $G = G_0$ or $G_0.2$. If $G = G_0$ then there are $3$ classes of $A_9$ subgroups in $\mathcal{S}$, with representatives $H_i$ for $i=1,2,3$. Since $|H_i||H_j|>|G|$, we see that every pair is non-regular. And using {\sc Magma}, one can check that all $4$-tuples are regular, and the only non-regular triples are $(H_i,H_i,H_i)$ for $i = 1,2,3$ (we know the latter are non-regular since $b(G,H_i) = 4$ by \cite[Theorem 1]{BGS}). And if $G = {\rm O}_8^{+}(2) = G_0.2$ then $\mathcal{S}$ comprises a unique class of $S_9$ subgroups and we have $b(G,H) = 4$ by \cite[Theorem 1]{BGS}, so $(H,H)$ and $(H,H,H)$ are non-regular, and $(H,H,H,H)$ is regular.

Now suppose $G_0 = {\rm P\O}_8^{+}(3)$. Here $G \leqs G_0.\la \gamma,\theta\ra = G_0.S_3$ and $H_0 = \O_8^{+}(2)$ (see \cite[Table 8.50]{BHR}), where $\gamma$ and $\theta$ denote graph automorphisms of order $2$ and $3$, respectively. 

To begin with, let us assume $G = G_0$, in which case there are $4$ classes of $\O_8^{+}(2)$ subgroups in $\mathcal{S}$. Since $|\O_8^{+}(2)|^2 > |G|$, we see that every pair is non-regular. And using {\sc Magma}, we see that every $4$-tuple is regular. Now assume $\tau = (H_1,H_2,H_3)$ is a triple of $\O_8^{+}(2)$ subgroups. We claim that $\tau$ is non-regular if and only if it is non-conjugate. 

To see this, first note that $b(G,H_i) = 3$ by \cite[Theorem 1]{BGS}, so every conjugate triple is regular. Now assume $\tau = (H_1,H_2,H_3)$ is non-conjugate. Using {\sc Magma}, we can construct a set $C$ of representatives of the $4$ classes of $\O_8^{+}(2)$ subgroups in $\mathcal{S}$:
{\small
\begin{verbatim}
G:=POmegaPlus(8,3);
M:=MaximalSubgroups(G);
C:=[M[i]`subgroup : i in [15..18]];
\end{verbatim}}
\noindent And then for every triple $\tau = (H_1,H_2,H_3)$ of subgroups in $C$ we can use the function \texttt{RegOrbits} (see Example \ref{ex:reg}) to calculate the number $r$ of regular orbits of $H_3$ on the Cartesian product $G/H_1 \times G/H_2$, noting that $\tau$ is regular if and only if $r \geqs 1$. In this way, it is straightforward to check that $\tau$ is non-regular if and only if it is non-conjugate, as claimed.

Similarly, if $G = G_0.\la \gamma\ra = G_0.2$ then $\mathcal{S}$ comprises two classes of ${\rm O}_8^{+}(2)$ subgroups and we can construct $G$ and a set $C$ of class representatives as in Example \ref{ex:con}(v). Then by working with the \texttt{RegOrbits} function as above,  we find that every pair and triple of subgroups in $\mathcal{S}$ is non-regular, whereas every $4$-tuple is regular. And if $G = G_0.3$ or $G_0.S_3$ then there is a unique class of subgroups in $\mathcal{S}$ and the main theorem of \cite{BGS} gives $b(G,H) = 4$, so $(H,H)$ and $(H,H,H)$ are non-regular, whereas $(H,H,H,H)$ is regular.

Finally, let us assume $G_0 = {\rm P\O}_{8}^{+}(5)$, so the possibilities for $H_0$ are listed in \eqref{e:o85}. As we showed in the proof of Lemma \ref{l:t2_82}, we have $\widehat{Q}(G,H)<1$ unless $G = G_0.S_3$ and $H = \O_8^{+}(2).S_3$. Therefore, we may assume $G = G_0.S_3$ and we note that each subgroup in $\mathcal{S}$ is conjugate to $H = \O_8^{+}(2).S_3$ or $K = (3 \times {\rm PGU}_3(5)).2$. By inspecting the proof of \cite[Lemma 5.5]{BGS}, we see that $b(G,H) = 2$ and thus $(H,H)$ is regular (note that the latter relies on a computation due to Breuer; see \cite{Breuer}).

So to complete the analysis of this case, it just remains to determine the regularity status of the pair $\tau = (H,K)$. In the notation of Lemma \ref{l:fpr}, we claim that
\[
\what{Q}(G,\tau) \leqs \frac{260161501}{409296875} < 1
\] 
and thus $\tau$ is regular. To see this, write 
\[
\what{Q}(G,\tau) = \a_0+\a_2+\a_3,
\]
where $\a_0$ is the contribution from elements in $G_0$, and $\a_r$ is the contribution from elements of order $r \in \{2,3\}$ in $G \setminus G_0$. We will work closely with the proof of Lemma \ref{l:t2_82}.

Let $x \in G_0$ be an element of prime order $r$ such that $x^G \cap H$ and $x^G \cap K$ are both non-empty. By considering the prime divisors of $|H \cap G_0|$ and $|K \cap G_0|$, we observe that $r \in \{2,3,5,7\}$.

Suppose $r = 2$. There is a unique class of involutions in $K \cap G_0 = {\rm PGU}_3(5)$ and they act on $V$ as $(-I_4,I_4)$, with plus-type eigenspaces. Therefore,  
\[
|x^G \cap K| = i_2({\rm PGU}_3(5)) = 525 = b_{1},\;\; |x^G| = \frac{|{\rm SO}_8^{+}(5)|}{4|{\rm SO}_{4}^{+}(5)|^2} = 42976171875 = c_1
\]
and by inspecting Table \ref{tab:aibi} we get $|x^G \cap H| = 58275 = a_{1}$. Similarly, if $r=3$ then $C_{G_0}(x)$ is of type ${\rm SO}_{2}^{-}(5) \times {\rm GU}_3(5)$ or ${\rm SO}_{4}^{+}(5) \times {\rm GU}_2(5)$, and the contribution to $\what{Q}(G,\tau)$ is given by $a_2b_2c_2^{-1} + a_3b_3c_3^{-1}$, where 
\[
\begin{array}{lll}
a_2 = 89600, &  b_2 = 15500, & c_2 = 2619500000000 \\
a_3 = 268800, & b_3 = 1050, & c_3 = 3438093750000.
\end{array}
\]
Elements in $H \cap G_0$ of order $5$ have Jordan form $(J_4^2)$ or $(J_5,J_1^3)$ on $V$, whereas those in $K \cap G_0$ have Jordan form $(J_3,J_2^2,J_1)$ or $(J_5,J_3)$, whence there is no contribution to $\what{Q}(G,\tau)$ from elements of order $5$. Finally, there is a unique conjugacy class of order $7$ elements in $G_0$ and the corresponding contribution is $a_4b_4c_4^{-1}$, where
\[
a_4 = i_7(\O_8^{+}(2)) = 24883200,\;\; b_4 = i_7({\rm PGU}_3(5)) = 36000,\;\; c_4 = 47151000000000000.
\]
Bringing this together, we conclude that 
\[
\a_0 = \sum_{i=1}^4 a_ib_ic_i^{-1} = \frac{549751}{409296875}.
\]

Now let us turn to $\a_2$. The group ${\rm PGU}_3(5).2$ has a unique class of involutory graph automorphisms and they act on $V$ as $(-I_5,I_3)$. Then by consulting the proof of Lemma \ref{l:t2_82}, we deduce that 
\[
\a_2  = a_5b_5c_5^{-1} = \frac{378}{50375},
\]
where $a_5 = 113400$, $b_5 = 3150$ and 
\[
c_5 = 3\cdot \frac{|{\rm SO}_8^{+}(5)|}{2|{\rm SO}_3(5)||{\rm SO}_5(5)|} = 47604375000.
\]

Finally, let $x \in G \setminus G_0$ be a triality graph automorphism. Once again, by inspecting the proof of Lemma \ref{l:t2_82}, we deduce that $\a_3 \leqs 2(a_6b_6c_6^{-1} + a_7b_7c_7^{-1})$, 
where
\[
a_6 = \frac{|\O_8^{+}(2)|}{|G_2(2)|} = 14400,\;\; c_6 = \frac{|G_0|}{|G_2(5)|} = 1521000000,
\]
\[
a_7 = \frac{|\O_8^{+}(2)|}{|{\rm PGU}_3(2)|} = 806400,\;\; c_7 = \frac{|G_0|}{|{\rm PGU}_3(5)|} = 23575500000000,
\]
and $b_6+b_7 \leqs 33102$. We then check that this upper bound on $\a_3$ is maximal when $b_6 = 33102$ and $b_7 = 0$, which means that
\[
\a_3 \leqs 2a_6c_6^{-1}\cdot 33102 = \frac{66204}{105625}.
\]

Finally, we now conclude that
\[
\what{Q}(G,\tau) = \a_0 + \a_2 + \a_3 \leqs \frac{549751}{409296875} + \frac{378}{50375} + \frac{66204}{105625} = \frac{260161501}{409296875}
\]
as claimed. The result follows.
\end{proof}

\begin{lem}\label{l:n7}
the conclusion of Theorem \ref{t:main} holds if $n = 7$. 
\end{lem}

\begin{proof}
Here $G_0 = \O_7(q)$ and either $H_0 = G_2(q)$, ${\rm Sp}_6(2)$ or $A_9$ (with $q=3$ in the latter case).

First assume $q = 3$, in which case $G = G_0$ and there are $6$ classes of subgroups in $\mathcal{S}$, with representatives $H_1,H_2$ (isomorphic to $G_2(3)$), $K_1,K_2$ (isomorphic to ${\rm Sp}_6(2)$) and $L_1,L_2$ (isomorphic to $S_9$). With the aid of {\sc Magma}, it is easy to check that every pair of subgroups in $\mathcal{S}$ is non-regular and every $4$-tuple is regular. As before, we can use the function \texttt{RegOrbits} (see Example \ref{ex:reg}) to determine the regularity status of every triple of subgroups in $\mathcal{S}$, but the result is more complicated to state. 

First we observe that every triple of the form $(H_i,H_j,H_k)$ is non-regular (this is recorded in parts (iii)(a) and (iv)(a) of Theorem \ref{t:main}). And the only non-regular triples with an $S_9$ component are $(H_1,H_1,L_1)$ and $(H_2,H_2,L_2)$, up to a suitable choice of labelling of the class representatives. Similarly, we find that $(K_1,K_1,K_2)$ and $(K_2,K_2,K_1)$ are non-regular, and so are the triples $(H_i,H_j,K_k)$ for all $i,j,k$. In addition, $(H_i,K_1,K_2)$ and $(H_i,K_i,K_i)$ are non-regular for $i = 1,2$ (again, up to a suitable choice of labelling). All other triples are regular.

Now assume $q \geqs 5$. If $q \ne p$ then $H_0 = G_2(q)$ is the only possibility (see \cite[Table 8.40]{BHR}) and there are two conjugacy classes of such subgroups in $\mathcal{S}$. If $\tau = (H_1,H_2,H_3)$ is a triple of such subgroups, then 
$|G:H_i| = \frac{1}{2}q^3(q^4-1)$ and we calculate that 
\[
|G| \geqs \frac{1}{2}q^9(q^2-1)(q^4-1)(q^6-1) > \frac{1}{8}q^9(q^4-1)^3 = |G:H_i|^3,
\]
so $G$ does not have a regular orbit on $G/H_1 \times G/H_2 \times G/H_3$ and thus every triple (and hence every pair) is non-regular (in agreement with parts (iii)(a) and (iv)(a) of Theorem \ref{t:main}). On the other hand, the main theorem of \cite{A} implies that every $4$-tuple is regular (see the proof of \cite[Proposition 6.3]{A} for the details). 

For the remainder, we may assume $q = p \geqs 5$ and $G = G_0$. There are $4$ classes of subgroups comprising $\mathcal{S}$, with representatives $H_1,H_2$ and $K_1,K_2$, where $H_i = G_2(q)$ and $K_i = {\rm Sp}_6(2)$ for $i=1,2$. By arguing as above, every $4$-tuple is regular, and every pair and triple comprising subgroups with socle $G_2(q)$ is non-regular. In addition, we note that $\what{Q}(G,K_i)<1$ for $i=1,2$ by Theorem \ref{t:hat}, so Lemma \ref{l:key} implies that every pair comprising subgroups with socle ${\rm Sp}_6(2)$ is regular. Therefore, it just remains to determine the regularity status of the pairs $(H_i,K_j)$ and the triples $(H_i,H_j,K_k)$ for $i,j,k \in \{1,2\}$.

If $q \in \{5,7\}$, then $|H_i||K_j| > |G|$ and so each pair $(H_i,K_j)$ is non-regular. And using {\sc Magma} (as in Example \ref{ex:con}(i)), we can construct $G$ and each $H_i, K_j$, and then by random search we can find elements $x,y \in G$ such that $H_i \cap H_j^x \cap K_k^y = 1$ for all $i,j,k \in \{1,2\}$. Therefore, every triple of the form $\tau = (H_i,H_j,K_k)$ is regular. Similarly, for $q=11$ we can check that every pair $(H_i,K_j)$ is regular.

To complete the proof, we may assume $q=p \geqs 13$ and we note that it suffices to show that $\tau = (H,K)$ is regular, where $H = G_2(q)$ and $K = {\rm Sp}_6(2)$. To do this, first note that $\pi(K) = \{2,3,5,7\}$ and write $\what{Q}(G,\tau) = \a+\b$, where $\a$ is the contribution from involutions. 

Now $H$ has a unique class of involutions, which act as $(-I_4,I_3)$ on the natural module for $G$ (see the proof of \cite[Lemma 2.13]{FPR4}, for example). Therefore,
\[
|x^G| \geqs \frac{|{\rm SO}_7(q)|}{2|{\rm SO}_{4}^{-}(q)||{\rm SO}_3(q)|} = \frac{1}{2}q^6(q^6-1) = f(q)
\]
for every involution $x \in H$ and thus
\[
\a \leqs i_2(H)i_2(K) \cdot f(q)^{-1} = 5103q^4(q^4+q^2+1) \cdot f(q)^{-1}.
\]

Now suppose $x \in G$ has order $r \in \{3,5,7\}$, so $x$ is semisimple since $q = p \geqs 13$. First assume $r = 5$. If $q^2 \equiv -1 \imod{5}$, then $i_5(H) = 0$, so let us assume $q^2 \equiv 1 \imod{5}$. Then $H$ has four classes of elements of order $5$, which gives 
\[
i_5(H) \leqs 4\cdot\frac{|G_2(q)|}{|{\rm GL}_2(q)|} < 5q^{10} = a_5
\]
(by \cite[Lemma 4.8]{BF}, no elements of order $5$ in $H$ are regular). In addition, $i_5(K) = 48384$ and there is a unique class of elements $x \in K$ of order $5$. By inspecting the character table of $K$ we deduce that
\[
|x^G| \geqs \frac{|{\rm SO}_7(q)|}{|{\rm SO}_3(q)||{\rm GU}_1(q)|^2}>\frac{1}{2}q^{16} = b_5.
\]
Similarly, for $r=7$ we have
\[
i_7(H) \leqs 6\cdot \frac{|G_2(q)|}{|{\rm GL}_2(q)|} + \frac{|G_2(q)|}{|{\rm GL}_1(q)|^2} < 2q^{12} = a_7
\]
and from the character table of $K$ we see that 
\[
|x^G| \geqs \frac{|{\rm SO}_7(q)|}{|{\rm GU}_1(q)|^3} > \frac{1}{2}q^{18} = b_7
\]
for all $x \in K$ of order $7$. We also note that $i_7(K) = 207360$.

We need a more precise estimate for the contribution to $\what{Q}(G,\tau)$ from elements of order $3$. To do this, set $\bar{V} = V \otimes k$, where $k$ is the algebraic closure of $\mathbb{F}_q$, and let $\omega \in k$ be a primitive cube root of unity. Now $K$ has $3$ classes of elements of order $3$, labelled \texttt{3A}, \texttt{3B} and \texttt{3C}. If $x \in K$ has order $3$, then by inspecting the character table of $K$ we deduce that
\[
\bar{V}{\downarrow}\,x = \left\{\begin{array}{ll}
(I_5,\omega,\omega^2) & \mbox{if $x \in \texttt{3A}$} \\
(I_1,\omega I_3, \omega^2I_3) & \mbox{if $x \in \texttt{3B}$} \\
(I_3,\omega I_2, \omega^2I_2) & \mbox{if $x \in \texttt{3C}$.} 
\end{array}\right.
\]
Similarly, if $x \in H$ has order $3$, then 
\[
\bar{V}{\downarrow}\,x = \left\{\begin{array}{ll}
(I_1,\omega I_3, \omega^2I_3) & \mbox{if $C_H(x) = {\rm SL}_3^{\e}(q)$} \\
(I_3,\omega I_2,\omega^2I_2) & \mbox{otherwise.}
\end{array}\right.
\]
Since $|\texttt{3B}| = 2240$ and $|\texttt{3C}| = 13440$, it follows that the precise contribution from order $3$ elements is
\[
\gamma = 2240 \cdot \frac{|G_2(q)|}{|{\rm SL}_3^{\e}(q)|} \cdot \frac{|{\rm GL}_3^{\e}(q)|}{|{\rm SO}_7(q)|} + 13440 \cdot \frac{|G_2(q)|}{|{\rm GL}^{\e}_2(q)|} \cdot \frac{|{\rm SO}_3(q)||{\rm GL}_2^{\e}(q)|}{|{\rm SO}_7(q)|},
\]
where $q \equiv \e \imod{3}$. 

By bringing together all of the above estimates, we conclude that 
\[
\b < \gamma + 48384a_5b_5^{-1} + 207360a_7b_7^{-1}.
\]
And by combining this with the above bound on $\a$, we are able to conclude that $\what{Q}(G,\tau)<1$ for all $q \geqs 13$.
\end{proof}

\begin{lem}\label{l:n5}
the conclusion of Theorem \ref{t:main} holds if $n \in \{5,6\}$. 
\end{lem}

\begin{proof}
If $n = 5$, then $G_0 = {\rm U}_5(2)$ and there is a unique class of subgroups in $\mathcal{S}$, which have socle $H_0 = {\rm L}_2(11)$, as recorded in \cite[Table 8.21]{BHR}. By the main theorem of \cite{BGS} we have $b(G,H) = 2$ and thus $(H,H)$ is regular.

Next assume $G_0 = {\rm U}_6(2)$, in which case $|G:G_0| \leqs 2$ and we have $H_0 = {\rm U}_4(3)$ or ${\rm M}_{22}$. To begin with, let us assume $G = G_0$ and observe that $\mathcal{S}$ comprises $3$ classes of subgroups isomorphic to ${\rm U}_4(3).2_2$, together with $3$ classes isomorphic to ${\rm M}_{22}$ (here the notation ${\rm U}_4(3).2_2$ indicates that this group contains involutory graph automorphisms of symplectic-type). Let us write $H_i$ and $K_i$ (with $i = 1,2,3$) for representatives of the respective classes. With the aid of {\sc Magma}, we find that all $4$-tuples are regular and all pairs are non-regular (for the latter claim, simply observe that $|H||K|>|G|$ for all $H,K \in \mathcal{S}$). In addition, all triples with three $H_i$ components are non-regular (since $|G| > |G:H_i|^3$), whereas every triple with two or more $K_i$ components is regular (this is easily checked by random search). Finally, suppose $\tau = (H_i,H_j,K_k)$. By working with the function \texttt{RegOrbits} (see Example \ref{ex:reg}), we find that exactly three such triples are regular (up to re-ordering); we can label the representatives so that $\tau$ is regular if and only if $(i,j,k) = (1,2,1)$, $(2,3,2)$ or $(3,1,3)$. 

Now assume $G = G_0.2$. Here each subgroup in $\mathcal{S}$ is conjugate to $H = {\rm U}_4(3).(2^2)_{122}$ or $K = {\rm M}_{22}.2$ (here the notation indicates that $H =H_0.\la x,y \ra$, where $x$ and $y$ are commuting involutions such that $x \in {\rm PGU}_4(3)$ and $y,xy$ are graph automorphisms of symplectic-type). Using {\sc Magma}, we check that every $5$-tuple is regular, and that $(H,H,H,H)$ is the only non-regular $4$-tuple, which corresponds to the special case recorded in part (ii) of Theorem \ref{t:main}. Every pair of subgroups is non-regular and with the aid of the \texttt{RegOrbits} function (see Example \ref{ex:reg}), the reader can readily check that $(H,H,H)$ and $(H,H,K)$ are the only non-regular triples.

To complete the proof of the lemma, we may assume $G_0 = {\rm Sp}_6(q)$ with $q$ even. If $q =2$ then $H = {\rm U}_3(3).2$ is the only possibility and there is a unique conjugacy class of such subgroups. By the main theorem of \cite{BGS} we have $b(G,H) = 4$, so $(H,H)$ and $(H,H,H)$ are non-regular, whereas $(H,H,H,H)$ is regular. Similarly, if $q \geqs 4$ is even, then $H_0 = G_2(q)$ is the only option and the same conclusion holds; this is the special case arising in parts (iii)(a) and (iv)(a) of Theorem \ref{t:main}. 
\end{proof}

\begin{lem}\label{l:n4}
the conclusion of Theorem \ref{t:main} holds if $n=4$.
\end{lem}

\begin{proof}
First assume $G_0 = {\rm U}_4(q)$ is a unitary group, in which case $q \in \{3,5\}$. For $q=5$, we can use {\sc Magma} to check that every pair of subgroups in $\mathcal{S}$ is regular, proceeding as in Example \ref{ex:con}(v) to construct $G$ and a set of representatives for the conjugacy classes of subgroups in $\mathcal{S}$.

Now suppose $G_0 = {\rm U}_4(3)$, so $H_0 = A_7$ or ${\rm L}_3(4)$. If $G = G_0$ then the subgroups in $\mathcal{S}$ are isomorphic to $A_7$ ($4$ classes) or ${\rm L}_3(4)$ ($2$ classes). With the aid of {\sc Magma}, we check that every $4$-tuple of subgroups in $\mathcal{S}$ is regular, whereas all pairs are non-regular. In addition, every triple with at least one $A_7$ component is regular, whereas all the triples with three ${\rm L}_3(4)$ components are non-regular, which we can verify by using the function \texttt{RegOrbits}.

Next assume $G = G_0.2_1 < {\rm PGU}_4(3)$. Here there are two classes in $\mathcal{S}$, both of which contain subgroups $H$   isomorphic to ${\rm L}_3(4).2_2$ (the notation indicates that $H$ contains involutory field automorphisms of $H_0$). As above, we find that all $4$-tuples are regular and all triples are non-regular. The case $G = G_0.2_3$ (extended by a graph automorphism of type ${\rm SO}_{4}^{-}(3)$) is similar. Here there are two classes of subgroups in $\mathcal{S}$; one contains subgroups isomorphic to ${\rm L}_3(4).2_1$, whereas the subgroups in the other class are isomorphic to ${\rm L}_3(4).2_3$ (so the former contain involutory graph-field automorphisms of $H_0$, and the latter contain involutory graph automorphisms). Once again, every $4$-tuple is regular and every triple is non-regular.

Next suppose $G = G_0.2_2$, so $G$ contains an involutory graph automorphism $x$ with $C_{G_0}(x)$ of type ${\rm Sp}_4(3)$. Here  there are two classes of subgroups in $\mathcal{S}$ isomorphic to $S_7$; all triples are regular and all pairs are non-regular. 

Finally, suppose $G = G_0.2^2 = G_0.\la \delta^2,\gamma\delta \ra$ in the notation of \cite[Table 8.11]{BHR}, noting that $\mathcal{S}$ is empty for all other almost simple groups with socle $G_0$. (Note that we can also write $G = G_0.(2^2)_{133}$ in this case, since all of the involutory graph automorphisms in $G$ are of type ${\rm SO}_{4}^{-}(3)$.) Here there are two classes of subgroups in $\mathcal{S}$ isomorphic to ${\rm L}_3(4).2^2$ and in the usual manner it is routine to check that every $4$-tuple is regular and all triples are non-regular.

Next assume $G_0 = {\rm PSp}_4(q)$. If $q$ is odd then we may assume $q \in \{5,7\}$, $G = G_0.2$ and $H = S_6$ (for $q = 5$) or $S_7$ (for $q=7$). For $q=5$ there is a unique class of subgroups in $\mathcal{S}$ and the main theorem of \cite{BGS} gives $b(G,H) = 2$, so $(H,H)$ is regular. Now assume $q=7$. Here there are two classes of subgroups in $\mathcal{S}$, represented by $H = S_7$ and $K = {\rm PGL}_2(7)$, and it is easy to check that every pair is regular.

Finally, suppose $G_0 = {\rm Sp}_4(q)$ with $q = 2^f$ even and $f \geqs 2$. If $f$ is even, then $\mathcal{S}$ is empty, so let us assume $f$ is odd. Then $H_0 = {}^2B_2(q)$ and there is a unique class of subgroups in $\mathcal{S}$. By the main theorem of \cite{BGS} we have $b(G,H) = 3$, so $(H,H,H)$ is regular and $(H,H)$ is non-regular, as recorded in part (iv)(b) of Theorem \ref{t:main}.
\end{proof}

\begin{lem}\label{l:n3}
the conclusion of Theorem \ref{t:main} holds if $n=3$.
\end{lem}

\begin{proof}
By inspecting \cite[Table 1]{BGS} and Table \ref{tab:new}, we see that 
\[
G_0 \in \{ {\rm L}_3(4), {\rm U}_3(3), {\rm U}_3(5)\}.
\]

First assume $G_0 = {\rm L}_3(4)$. If $G = G_0$ then there are $3$ classes of subgroups in $\mathcal{S}$, each comprising subgroups isomorphic to $A_6$, and using {\sc Magma} we find that all triples are regular and all pairs are non-regular. Next assume $G = G_0.2 = G_0.\la \theta \ra$. If $\theta$ is an involutory graph-field automorphism then there are $3$ classes of subgroups in $\mathcal{S}$ isomorphic to ${\rm M}_{10} = A_6.2$; every triple is regular and every pair is non-regular. And if $\theta$ is a field or graph automorphism, then there is a unique class of subgroups $H \in \mathcal{S}$ (with $H \cong S_6$ if $\theta$ is a field automorphism, otherwise $H \cong {\rm PGL}_{2}(9)$) and we note that $b(G,H) = 3$ by \cite[Theorem 1]{BGS}, so $(H,H)$ is non-regular and $(H,H,H)$ is regular. Finally, we may assume $G = G_0.2^2$ (in all other cases, $\mathcal{S}$ is empty). Here there is a unique class of subgroups $H \in \mathcal{S}$ with $H = {\rm Aut}(A_6) = A_6.2^2$ and we can use {\sc Magma} to check that $b(G,H) = 4$, so $(H,H)$ and $(H,H,H)$ are non-regular and $(H,H,H,H)$ is regular (see Remark \ref{r:main}(e)).

Next assume $G_0 = {\rm U}_3(3)$, so $H_0 = {\rm L}_2(7)$ and there is a unique class of subgroups in $\mathcal{S}$. If $G = G_0$ then $H = H_0$ and $b(G,H) = 3$, so $(H,H)$ is non-regular and $(H,H,H)$ is regular. Similarly, if $G = G_0.2$, then $H = H_0.2$ and $b(G,H) = 4$, so $(H,H)$ and $(H,H,H)$ are non-regular, and $(H,H,H,H)$ is regular. 

Finally, suppose $G_0 = {\rm U}_3(5)$. If $G = G_0$, then there are $6$ conjugacy classes of subgroups comprising $\mathcal{S}$, with representatives $H_i$ and $K_i$ (for $i = 1,2,3$), where $H_i \cong A_7$
and $K_i \cong A_6.2 = {\rm M}_{10}$. Just by considering orders, we see that every pair of subgroups in $\mathcal{S}$ is non-regular. And with the aid of {\sc Magma}, we check that every $4$-tuple is regular. The result for triples is more complicated. Firstly, every triple of $A_7$ subgroups is non-regular since $|G|>|G:H_i|^3$, whereas every triple with two or more ${\rm M}_{10}$ components is regular, so it remains to consider triples of the form $(A_7,A_7,{\rm M}_{10})$. Using the function \texttt{RegOrbits} (see Example \ref{ex:reg}), we find that there are precisely $3$ regular triples, up to ordering and conjugacy; we can label the $H_i$ and $K_i$ so that $(H_i,H_i,K_i)$ is regular for $i = 1,2,3$. 

The case $G = {\rm U}_3(5).2$ is similar. Here each subgroup in 
$\mathcal{S}$ is isomorphic to $S_7$, $A_6.2^2$ or ${\rm PGL}_2(7)$, and there is a unique class of each type. In the usual manner, we find that every $4$-tuple is regular and every pair is non-regular. In addition, the only non-regular triples are $(S_7,S_7,S_7)$ and $(S_7,S_7,A_6.2^2)$.
\end{proof}

\begin{lem}\label{l:n2}
the conclusion of Theorem \ref{t:main} holds if $n=2$.
\end{lem}

\begin{proof}
Here $G_0 = {\rm L}_2(q)$ and we may assume $q \in \{11,19,29,31\}$, so $G = G_0$ and $H = A_5$ (two classes). If $q \in \{29,31\}$, then we can use {\sc Magma} to check that every pair of subgroups in $\mathcal{S}$ is regular. On the other hand, if $q \in \{11,19\}$ then every pair of subgroups in $\mathcal{S}$ is non-regular. In addition, if $q = 19$ then every triple is regular, whereas $(H,H,K)$ and $(H,K,K)$ are the only non-regular triples when $q =11$, where $H$ and $K$ are representatives of the two classes of $A_5$ subgroups in $\mathcal{S}$. Finally, for $q=11$ we check that every $4$-tuple is regular.
\end{proof}

\vs

This completes the proof of Theorem \ref{t:main}.

\section{The tables}\label{s:tables}

\setcounter{table}{0}
\renewcommand{\thetable}{\Alph{table}}
\renewcommand{\theHtable}{final.\Alph{table}}

In this final section, we present Tables \ref{tab:3} and \ref{tab:2}.

\begin{rem}\label{r:tables}
Let us record some comments on the set-up and notation we use in Tables \ref{tab:3} and \ref{tab:2}.
\begin{itemize}\addtolength{\itemsep}{0.2\baselineskip}
\item[{\rm (a)}] Firstly, recall that each tuple $\tau$ is recorded up to re-ordering and conjugacy in $G$.
\item[{\rm (b)}] If $G$ has $\ell$ conjugacy classes of maximal subgroups in $\mathcal{S}$ isomorphic to a fixed almost simple group $H$, then we use the notation $H_1, \ldots, H_{\ell}$ for a set of representatives of these classes. In certain cases, the labels for these representatives have been chosen so that the conditions in Tables \ref{tab:3} and \ref{tab:2}  are convenient to state.
\item[{\rm (c)}] In the third row of Table \ref{tab:3}, we have $G_0 = {\rm P\O}_8^{+}(3)$ and $G = G_0.2_2$ denotes an extension of $G_0$ by an involutory graph automorphism of type $(-I_1,I_7)$, which is an index-two subgroup of ${\rm PO}_8^{+}(3) = G_0.2^2$. This notation is consistent with \cite{GAPCTL}. 
\item[{\rm (d)}] Similarly, for $H_0 = {\rm U}_4(3)$ we write $H_0.2_1$ for the unique index-two subgroup of ${\rm PGU}_4(3)$, and we use $H_0.2_2$ and $H_0.2_3$ to denote extensions of the form $H_0.\la \gamma\ra$, where $\gamma$ is an involutory graph automorphism of $H_0$. In the first case, the centraliser of $\gamma$ is of type ${\rm Sp}_4(3)$, while in the second it is of type ${\rm SO}_4^{-}(3)$. Similarly, ${\rm Aut}(H_0)$ has two conjugacy classes of subgroups of the form $H_0.2^2$, with representatives denoted by $H_0.(2^2)_{122}$ and $H_0.(2^2)_{133}$; here the former subgroups contain symplectic-type graph automorphisms, while the latter do not.
\item[{\rm (e)}] Let $H_0 = {\rm L}_3(4)$. Up to conjugacy in ${\rm Aut}(H_0)$, there are three almost simple groups of the form $H_0.2$, with representatives denoted by $H_0.2_i$ for $i=1,2,3$. Here $H_0.2_1$ and $H_0.2_2$ contain involutory graph-field and field automorphisms, respectively, while $H_0.2_3$ contains involutory graph automorphisms.
\end{itemize}
\end{rem}

\begin{table}
\[
\begin{array}{lll} \hline
G & \tau  & \mbox{Conditions} \\ \hline
{\rm O}_{10}^{-}(2) & (H,H,H) & H = S_{12} \\
{\rm P\O}_{8}^{+}(3) & (H_i,H_j,H_k) & H = \O_8^{+}(2), \,  i,j,k \in \{1,2,3,4\},\, i \ne j \\
{\rm P\O}_{8}^{+}(3).2_2 & (H_i, H_j, H_k) & H ={\rm O}_8^{+}(2),\, i,j,k \in \{1,2\} \\
{\rm P\O}_{8}^{+}(3).X & (H,H,H) & H = \O_8^{+}(2).X,\, 
X \in \{C_3,S_3\} \\
\O_{8}^{+}(2) & (H_i,H_i,H_i) & H = A_9, \, i \in \{1,2,3\} \\
{\rm O}_8^{+}(2) & (H,H,H) & H = S_9 \\
\O_7(3) & (H_i,H_j,K_k) & H = G_2(3),\, K = {\rm Sp}_6(2), \, i,j,k \in \{1,2\} \\ 
& (H_i,H_i,K_i) & H = G_2(3), \, K = S_9, \, i \in \{1,2\} \\  
& (H_i,K_1,K_2) & H = G_2(3),\, K = {\rm Sp}_6(2), \, i \in \{1,2\} \\
& (H_i,K_i,K_i) & H = G_2(3),\, K = {\rm Sp}_6(2), \, i \in \{1,2\} \\
& (H_i,H_j,H_j) & H = {\rm Sp}_6(2),\, i,j \in \{1,2\},\, i \ne j \\ 
{\rm U}_6(2) & (H_i, H_j, H_k) & H = {\rm U}_4(3).2_2,\,  i,j,k \in \{1,2,3\} \\
& (H_i, H_j, K_k) & H = {\rm U}_4(3).2_2,\, K = {\rm M}_{22},\, i,j,k \in \{1,2,3\}, \\
& & (\{i,j\},k) \ne (\{1,2\},1),(\{2,3\},2),(\{1,3\},3) \\
{\rm U}_6(2).2 & (H, H, K) & H = {\rm U}_4(3).(2^2)_{122}, \, K \in \{{\rm U}_4(3).(2^2)_{122}, {\rm M}_{22}.2\} \\
{\rm Sp}_6(2) & (H,H,H) & H = {\rm U}_3(3).2 \\
{\rm U}_4(3) & (H_i,H_j,H_k) & H = {\rm L}_3(4), \, i,j,k \in \{1,2\} \\
{\rm U}_4(3).2_1 & (H_i,H_j,H_k) & H = {\rm L}_3(4).2_2,\, i,j,k \in \{1,2\} \\
{\rm U}_4(3).2_3 & (H,K,L) & H,K,L \in \{ {\rm L}_3(4).2_1, {\rm L}_3(4).2_3\} \\
{\rm U}_4(3).(2^2)_{133} & (H_i,H_j,H_k) & H = {\rm L}_3(4).2^2,\, i,j,k \in \{1,2\}  \\
{\rm L}_3(4).2^2 & (H,H,H) & H = A_6.2^2 \\
{\rm U}_3(5) & (H_i,H_j,H_k) & H = A_7,\,  i,j,k \in \{1,2,3\} \\
& (H_i,H_j,K_k) & H = A_7, \, K = {\rm M}_{10}, \, i,j,k \in \{1,2,3\},\, (i,j) \ne (k,k) \\
{\rm U}_3(5).2 & (H,H,K) & H = S_7, \, K \in \{ S_7, A_6.2^2\} \\
{\rm U}_3(3).2 & (H,H,H) & H = {\rm PGL}_2(7)  \\
{\rm L}_2(11) & (H_1,H_2,H_i) & H = A_5,\, i \in\{1,2\} \\
\hline
\end{array}
\]
\caption{Non-regular triples}
\label{tab:3}
\end{table}

\begin{table}
\[
\begin{array}{lll} \hline
G & \tau & \mbox{Conditions} \\ \hline
{\rm O}_{14}^{+}(2) & (H,H) & H = S_{16} \\
\O_{12}^{-}(2).c & (H,H) & H = A_{13}.c, \, c \in \{1,2\} \\
\O_{10}^{-}(2) & (H,H) & H = A_{12} \\
{\rm O}_{10}^{-}(2) & (H,K) & H = S_{12}, \, K \in \{S_{12},{\rm M}_{12}.2\} \\
{\rm Sp}_8(2) & (H,H) & H = S_{10} \\
{\rm P\O}_{8}^{+}(3) & (H_i,H_j) & H =\O_8^{+}(2), \, i,j \in \{1,2,3,4\} \\
{\rm P\O}_{8}^{+}(3).2_2 & (H_i,H_j) & H = {\rm O}_8^{+}(2),\, i,j \in \{1,2\} \\
{\rm P\O}_{8}^{+}(3).X & (H,H) & H = \O_8^{+}(2).X,\, 
X \in \{C_3,S_3\} \\
\O_{8}^{+}(2) & (H_i,H_j) & H = A_9, \, i,j \in \{1,2,3\} \\
{\rm O}_8^{+}(2) & (H,H) & H = S_9 \\
\O_7(q) & (H_i,K_j) & H = G_2(q), \, K = {\rm Sp}_6(2), \, i,j \in \{1,2\},\, q \in \{5,7\} \\
\O_7(3) & (H_i,K_j) & H,K \in \{ G_2(3), {\rm Sp}_6(2), S_9 \},\,i,j \in \{1,2\}  \\
{\rm U}_6(2) & (H_i,K_j) & H,K \in \{ {\rm U}_4(3).2_2, {\rm M}_{22} \},\, i,j \in \{1,2,3\} \\
{\rm U}_6(2).2 & (H,K) & H,K \in \{ {\rm U}_4(3).(2^2)_{122}, {\rm M}_{22}.2 \} \\
{\rm Sp}_6(2) & (H,H) & H = {\rm U}_3(3).2 \\
{\rm U}_4(3) & (H_i,H_j) & H = {\rm L}_3(4),\, i,j \in \{1,2\} \\
& (H_i,H_j) & H = A_7,\,  i,j \in \{1,2,3,4\} \\
& (H_i,K_j) & H ={\rm L}_3(4), \, K = A_7, \, i \in \{1,2\},\,j \in \{1,2,3,4\} \\
{\rm U}_4(3).2_1 & (H_i,H_j) & H = {\rm L}_3(4).2_2,\,  i,j \in \{1,2\} \\
{\rm U}_4(3).2_2 & (H_i,H_j) & H = S_7, \, i,j \in \{1,2\} \\
{\rm U}_4(3).2_3 & (H,K) & H,K \in \{ {\rm L}_3(4).2_1, {\rm L}_3(4).2_3 \} \\
{\rm U}_4(3).(2^2)_{133} & (H_i,H_j) &  H = {\rm L}_3(4).2^2, \, i,j \in \{1,2\} \\
{\rm L}_3(4) & (H_i,H_j) & H = A_6,\, i,j \in \{1,2,3\} \\
{\rm L}_3(4).2_1 & (H_i,H_j) & H = {\rm M}_{10},\, i,j \in \{1,2,3\} \\
{\rm L}_3(4).2_2 & (H,H) & H = S_6 \\
{\rm L}_3(4).2_3 & (H,H) & H = {\rm PGL}_2(9) \\
{\rm L}_3(4).2^2 & (H,H) & H = A_6.2^2 \\
{\rm U}_3(5) & (H_i,K_j) & H,K \in \{{\rm M}_{10},A_7\},\, i,j \in \{1,2,3\} \\
{\rm U}_3(5).2 & (H,K) & H,K \in \{A_6.2^2, S_7, {\rm PGL}_2(7)\} \\
{\rm U}_3(3) & (H,H) & H = {\rm L}_2(7) \\
{\rm U}_3(3).2 & (H,H) & H = {\rm PGL}_2(7) \\
{\rm L}_2(q) & (H_i,H_j) & H = A_5, \, i,j \in \{1,2\},\, q \in \{11,19\} \\
\hline
\end{array}
\]
\caption{Non-regular pairs}
\label{tab:2}
\end{table}

\end{document}